\documentclass[a4paper,11pt]{amsart}

\usepackage[
  margin=30mm,
  marginparwidth=25mm,     
  marginparsep=2mm,       
  bottom=25mm,
  ]{geometry}

\usepackage[latin1]{inputenc}
\usepackage[T1]{fontenc}
\usepackage[english]{babel}
\usepackage{amsmath,amssymb,amsthm}
\usepackage{latexsym}
\usepackage{delarray}
\usepackage{bbm}
\usepackage{hyperref}
\usepackage[pdftex,usenames,dvipsnames]{color}
\usepackage{bbding,trfsigns}
\usepackage{datetime}
\usepackage{mathrsfs}
\usepackage{enumitem}
\usepackage{tikz}
\usepackage{gensymb}

\newtheorem{Th}{Theorem}[section]
\newtheorem{Prop}[Th]{Proposition}
\newtheorem{Lem}[Th]{Lemma}
\newtheorem{Cor}[Th]{Corollary}

\newcommand{\C}{\mathbb{C}}

\newcommand{\N}{\mathbb{N}}
\newcommand{\R}{\mathbb{R}}

\newcommand{\Rd}{\mathbb{R}^{d}}
\newcommand{\Rdd}{\mathbb{R}^{d+1}_+}
\newcommand{\Z}{\mathbb{Z}}

\newcommand{\G}{\Gamma}

\newcommand{\Aa}{\mathcal{A}}

\newcommand{\HH}{\mathcal{H}}
\newcommand{\Hpq}{\mathcal{H}^{p,q}(\Rd)}
\newcommand{\HHpq}{\mathbb{H}^{p,q}(\Rdd)}
\newcommand{\HHpqt}{\mathbb{H}^{p,q}_{\Theta}(\Rdd)}
\newcommand{\Hpqt}{\mathcal{H}^{p,q}_{\Theta}(\Rdd)}

\newcommand{\LL}{\mathcal{L}}
\newcommand{\MM}{\mathcal{M}}
\newcommand{\OO}{\mathcal{O}}

\newcommand{\Ss}{\mathcal{S}}

\newcommand{\RR}{\mathcal{R}}

\newcommand{\1}{\mathbbm{1}}

\newcommand{\eps}{\varepsilon}

\DeclareMathOperator{\supp}{supp}
\DeclareMathOperator{\rank}{rank}

\allowdisplaybreaks 

\title
[Riesz transforms, Cauchy-Riemann systems and amalgam Hardy spaces]
{Riesz transforms, Cauchy-Riemann systems and amalgam Hardy spaces}

\author[A. Assaubay]{A. Assaubay}
\author[J. J. Betancor]{J. J. Betancor}
\author[A. J. Castro]{A. J. Castro}
\author[J.C. Fari\~na]{J.C. Fari\~na}

\address{Jorge J. Betancor, Juan C. Fari\~na\newline
	Departamento de An\'alisis Matem\'atico, Universidad de La Laguna,\newline
	Campus de Anchieta, Avda. Astrof\'isico S\'anchez, s/n,\newline
	38721 La Laguna (Sta. Cruz de Tenerife), Spain}
\email{jbetanco@ull.es, jcfarina@ull.es}

\address{Al-Tarazi Assaubay, Alejandro J. Castro \newline
       Department of  Mathematics, Nazarbayev University, \newline
		010000 Astana, Kazakhstan}
\email{tarazi.assaubay@nu.edu.kz, alejandro.castilla@nu.edu.kz}

\address{\newline}

\thanks{The second and the fourth authors were partially supported by MTM2016-79436-P. The third author was supported by Nazarbayev University Social Policy Grant.}



\begin{document}



\begin{abstract}
In this paper we study Hardy spaces $\Hpq$, $0<p,q<\infty$, modeled over amalgam spaces $(L^p,\ell^q)(\Rd)$.
We characterize $\Hpq$ by using first order classical Riesz transforms and compositions of first order Riesz transforms depending on the values of the exponents $p$ and $q$. Also, we describe the distributions in $\Hpq$ as the boundary values of solutions of harmonic and caloric Cauchy-Riemann systems. We remark that caloric Cauchy-Riemann systems involve fractional derivative in the time variable.
Finally we characterize the functions in $L^2(\Rd) \cap \Hpq$ by means of Fourier multipliers $m_\theta$ with symbol $\theta(\cdot/|\cdot|)$, where $\theta \in C^\infty(\mathbb{S}^{d-1})$ and $\mathbb{S}^{d-1}$ denotes the unit sphere in $\Rd$.
\end{abstract}

\maketitle

\section{Introduction}

In this paper we study Hardy spaces $\Hpq$, $0<p,q<\infty$, modeled over amalgam spaces $(L^p,\ell^q)(\Rd)$. These spaces have been investigated recently by Abl\'e and Feuto (\cite{AF1}, \cite{AF2} and \cite{AF3}; see also \cite{ZYYW}).\\

Amalgam spaces were first defined by Wiener in 1926. He considered the particular spaces $(L^1,\ell^2)(\R)$ and $(L^2,\ell^\infty)(\R)$ in \cite{W1}; $(L^1,\ell^\infty)(\R)$ and $(L^\infty,\ell^1)(\Rd)$ in \cite{W2}. For every $0<p,q<\infty$ the amalgam space $(L^p,\ell^q)(\R)$ consists of all those $f \in L^p_{loc}(\R)$ such that
\begin{equation*}\label{eq:I1}
\|f\|_{p,q}
:= \Big( \sum_{n \in \Z} \Big( \int_{n}^{n+1} |f(x)|^p \, dx \Big)^{q/p} \Big)^{1/q}<\infty.
\end{equation*}
When $p$ or $q$ is infinity, the usual adjustments should be done.
As it can be observed in the definition of $\|\cdot\|_{p,q}$ it contains information about the local-$L^p$ and the global-$\ell^q$ properties of the functions, in contrast with the standard Lebesgue $L^p$-spaces. Notice also that $L^p(\R)=(L^p,\ell^p)(\R)$, $0<p<\infty$.\\

The first systematic study of the space $(L^p,\ell^q)(\R)$ was undertaken in 1975 by Holland (\cite{Hol}). Feichtinger (\cite{Fei}) generalized the definition of amalgams to Banach spaces. Amalgams have appeared in various areas of analysis: Tauberian theorems, almost periodic functions, Fourier multipliers, domain and range of Fourier transform, algebras and modules \dots A complete survey on amalgam spaces was written by Fournier and Stewart (\cite{FoSt}).\\

For every $d \in \N$, $d \geq 2$, and $0<p,q<\infty$ we say that a function $f \in L^p_{loc}(\Rd)$ is in the amalgam space $(L^p,\ell^q)(\Rd)$ when
$$\|f\|_{p,q}
:=\Big\| \|\1_{B(y,1)}f\|_p \Big\|_q < \infty.$$
Here $\1_E$ represents, as usual, the characteristic function over the set $E$.
We can also write a "discrete" equivalent expression for $\| \cdot \|_{p,q}$ (\cite[eq. (2.1)]{AF1}).\\

Following the ideas of Stein and Weiss (\cite{StWe2}) and Fefferman and Stein (\cite{FS}) several generalizations of Hardy spaces $\HH^p(\Rd)$ have been considered. One possible way of defining new Hardy spaces is to replace the Lebesgue space $L^p(\Rd)$ by other function space $X(\Rd)$ on $\Rd$. When this is made we say that a Hardy space modeled over $X(\Rd)$ have been defined. In the last years, Hardy spaces modeled over several spaces (Orlicz, Lorentz, Musielak-Orlicz, Lebesgue spaces with variable exponents, \dots) have been considered (see, for instance, \cite{AT},
\cite{CW},
\cite{LYJ},
\cite{NS},
\cite{SHYY} and
\cite{YLK}).\\

Hardy spaces modeled on amalgams were defined by Abl\'e and Feuto (\cite{AF1}) as follows. Let $0<p,q<\infty$. Assume that $\varphi \in S(\Rd)$ with $\int_{\Rd} \varphi(x) \, dx \neq 0$. A tempered distribution $f$ on $\Rd$ is said to be in the Hardy space $\Hpq$ when the maximal function $M_\varphi(f)$ of $f$ defined by
$$M_\varphi(f)(x)
:=\sup_{t>0} |(f * \varphi_t)(x)|, \quad x \in \Rd,$$
is in $(L^p,\ell^q)(\Rd)$. Here, $\varphi_t(x):=t^{-n}\varphi(x/t)$. We consider, for every $f\in \Hpq$,
$$
\|f\|_{\Hpq}:=\|M_\varphi(f)\|_{p,q}.
$$
$\Hpq$ is a quasi--Banach space when it is endowed with the quasi-norm $\|.\|_{\Hpq}$. In particular, $\Hpq=(L^p,\ell^q)(\mathbb{R}^d)$, provided that $1<p,q<\infty$ (\cite[Theorem 3.2.1]{AF1}).\\

In \cite[Theorem 3.7]{AF1} $\Hpq$ is characterized by using nontangential and grand maximal functions and it is proved that $\Hpq$ is actually independent of the choice of the function $\varphi$ use to define the maximal function.\\

The distributions in $\Hpq$ are bounded, that is, $f * \varphi \in L^\infty(\Rd)$, for every $\varphi \in S(\Rd)$ (\cite[(3.12) and (3.13)]{AF1}). This property allows us to define, for every $f \in \Hpq$, the convolution $f * P_t$ of $f$ with the Poisson kernel
$$P(x)
:= \frac{\G(\frac{d+1}{2})}{\pi^{(d+1)/2}} \frac{1}{(1+|x|^2)^{(d+1)/2}}, \quad x \in \Rd,$$
according to \cite[p. 89--90]{St1}. If $f \in S(\Rd)'$ is bounded, the function
$$u(x,t):=(f*P_t)(x), \quad x \in \Rd, \, t>0,$$
is harmonic and $f \in \Hpq$ if, and only if, the maximal function
$$u_+(x):=\sup_{t>0}|u(x,t)|, \quad x \in \Rd,$$
is in $(L^p,\ell^q)(\Rd)$ (\cite[p. 16]{AF1}).\\

Atomic characterizations of the distributions in $\Hpq$ were established in \cite[Theorems 4.3, 4.4 and 4.6]{AF1}.
Dual spaces of $\Hpq$ and the boundedness of certain pseudo-differential operators, Calder\'on-Zygmund operators and Riesz potentials in these Hardy spaces were studied in \cite{AF2} and \cite{AF3}.\\

Recently, Sawano, Ho, D. Yang and S. Yang (\cite{SHYY}) have developed a real variable theory of Hardy spaces modeled over a general class of function spaces called ball quasi-Banach functions spaces. A quasi-Banach space $X$ of measurable functions on $\Rd$ is a ball quasi-Banach function space on $\Rd$ when the following properties hold:
\begin{itemize}
\item[$i)$] $\|f\|_X=0$ implies that $f=0$;
\item[$ii)$] If $|g|\leq |f|$ and $f,g \in X$, then $\|g\|_X \leq \|f\|_X$;
\item[$iii)$] If $\{f_m\}_{m=1}^\infty$ is an increasing sequence in $X$, $f \in X$ and
$\displaystyle \lim_{m \to \infty} f_m(x)=f(x)$, a.e. $x \in \Rd$, then $\|f_m\|_X \longrightarrow \|f\|_X$, as $m \to \infty$;
\item[$iv)$] For every $x \in \Rd$ and $r>0$, $\1_{B(x,r)} \in X$.
\end{itemize}
According to \cite[Lemma 2.28]{ZYYW} amalgam spaces
$(L^p,\ell^q)(\Rd)$, $0<p,q<\infty$ are quasi-Banach function spaces on $\Rd$. Then, Hardy spaces modeled over amalgams studied in \cite{AF1} can be seen as special cases of the ones considered in \cite{SHYY}. In \cite{ZYYW}, Zhang, Yang, Yuan and Wang have defined Orlicz-slice Hardy spaces on $\Rd$. This new scale of Hardy spaces contains also the Hardy-amalgam spaces $\Hpq$ of Abl\'e and Feuto (\cite{AF1}).\\

In this paper we complete the theory of Hardy spaces over amalgams. We establish for $\Hpq$ some new results that are not contained in previous investigations \cite{SHYY} and \cite{ZYYW}.\\

If $f \in \Hpq$, then, for every $\phi \in S(\Rd)$,
$f * \phi \in (L^{\mu p},\ell^{\mu q})(\Rd)$, for every $\mu \geq 1$ (see Section \ref{Sect:2}). This fact allows us to define Riesz transformations $R_j$, $j=1, \dots, d$, on $\Hpq$ following the ideas developed in
\cite[p. 123]{St1} for the classical Hardy space.
If $0<p,q<\infty$, we say that a distribution $f \in S(\Rd)'$ is $(p,q)$-restricted at infinity when there exists $\mu_0 \geq 1$ such that $f * \phi \in (L^{\mu p},\ell^{\mu q})(\Rd)$ for every $\phi \in S(\Rd)$ and $\mu \geq \mu_0$. In the next result we characterize the distributions in $\Hpq$ by using Riesz transforms.

\begin{Th}\label{Th:1.1} \quad
\begin{itemize}
\item[$i)$] Let $(d-1)/d < min\{p,q\}<\infty$, $f \in S(\Rd)'$ and $\phi \in S(\Rd)$ such that $\int_{\Rd} \phi(x) dx \neq 0$. Then, $f \in \HH^{p,q}(\Rd)$ if, and only if, $f$ is $(p,q)$-restricted at infinity and
$$\sup_{\eps >0} \Big( \|f * \phi_\eps \|_{p,q} + \sum_{j=1}^d \|(R_j f) * \phi_\eps\|_{p,q}\Big) < \infty.$$
Moreover, the quantities
$\|f\|_{\HH^{p,q}(\Rd)}$ and
$$\sup_{\eps >0} \Big( \|f * \phi_\eps \|_{p,q} + \sum_{j=1}^d \|(R_j f) * \phi_\eps\|_{p,q}\Big)$$
are equivalent.

\item[$ii)$] Assume that $m \in \N$, $(d-1)/(d+m-1) < min\{p,q\}<\infty$, $f \in S(\Rd)'$ and $\phi \in S(\Rd)$ such that $\int_{\Rd} \phi(x) dx \neq 0$. Then, $f \in \HH^{p,q}(\Rd)$ if, and only if, $f$ is $(p,q)$-restricted at infinity and
$$\sup_{\eps >0} \Big( \|f * \phi_\eps \|_{p,q} + \sum_{k=1}^m \sum_{j_1, \dots, j_k=1}^d \|R_{j_1} \dots  R_{j_k}(f) * \phi_\eps\|_{p,q}\Big) < \infty.$$
Moreover, the quantities $\|f\|_{\HH^{p,q}(\Rd)}$ and
$$\sup_{\eps >0} \Big( \|f * \phi_\eps \|_{p,q} + \sum_{k=1}^m \sum_{j_1, \dots, j_k=1}^d \|R_{j_1} \dots  R_{j_k}(f) * \phi_\eps\|_{p,q}\Big)$$
are equivalent.
\end{itemize}
\end{Th}

In order to prove Theorem \ref{Th:1.1} we follow the strategy developed in \cite[p. 123--124]{St1}. These arguments have been also used to characterize Hardy spaces modeled over other spaces (see, for instance,
\cite{CCYY},
\cite{WZZJ} and
\cite{YZN}). We need to make some modifications in order to adapt the method in \cite[p. 123--124]{St1} to our amalgam setting.\\

In \cite[p. 120]{St1} the Hardy space $\HH^p(\Rd)$ is characterized in terms of systems of conjugate harmonic functions. We say that a vector
$F:=(u_1, \dots, u_{d+1})$ of $d+1$ functions in $\Rdd:=\mathbb{R}^d\times (0,\infty)$ satisfies the generalized Cauchy-Riemann equations, shortly $F \in CR(\Rdd)$, when
$$\frac{\partial u_j}{\partial x_k} = \frac{\partial u_k}{\partial x_j}, \quad 1 \leq j,k \leq d+1,
\qquad \text{and} \qquad \sum_{j=1}^{d+1} \frac{\partial u_j}{\partial x_j} = 0.$$
Here and in the sequel we identify $x_{d+1}=t$.\\

In order to prove Theorem \ref{Th:1.1} we previously need to establish the following characterization of $\Hpq$.

\begin{Th}\label{Th:1.2}
Suppose that $u$ is a harmonic function in $\Rdd$.
\begin{itemize}
\item[$i)$] There exists $f \in \Hpq$ such that
$u(x,t)=P_t(f)(x)$, $x \in \Rd$ and $t>0$.

\item[$ii)$] $u^* \in (L^p,\ell^q)(\Rd)$, where
$$ u^*(x)
:= \sup_{|x-y|<t, \ t>0} |u(y,t)|, \quad x \in \Rd.$$

\item[$iii)$] There exists a harmonic vector
$F:=(u_1, \dots, u_{d+1}) \in CR(\Rdd)$ such that $u_{d+1}:=u$ and
$\displaystyle \sup_{t>0} \||F(\cdot,t)|\|_{p,q}<\infty.$
\end{itemize}
Moreover
\begin{itemize}
\item if $0<p<q<\infty$, $i) \Leftrightarrow ii)$ and
$\|f\|_{\Hpq}$ and $\|u^*\|_{p,q}$ are equivalent; and
\item when $(d-1)/d < min\{p,q\}<\infty$,
$i) \Leftrightarrow ii) \Leftrightarrow iii)$ with equivalent quantities $\|f\|_{\Hpq}$, $\|u^*\|_{p,q}$ and
$\displaystyle \sup_{t>0} \||F(\cdot,t)|\|_{p,q}.$
\end{itemize}
\end{Th}

Kochneff and Sagher (\cite{KS}) considered temperature Cauchy-Riemann equations in $\R^2_+$. Following the idea of conjugacy in \cite{KS} (see also \cite{G}) Guzm\'an-Partida and P\'erez-Esteva (\cite{GP}) defined temperature Cauchy-Riemann equations in $\Rdd$. A vector $F=(u_1, \dots, u_{d+1})$ of $d+1$ functions in $\R^{d+1}$ is said to satisfy the generalized temperature Cauchy-Riemann equations, shortly $F \in TCR(\Rdd)$, when
\begin{itemize}
\item[$a)$] $\displaystyle \sum_{j=1}^{d} \frac{\partial u_j}{\partial x_j} = i \, \partial_t^{1/2} u_{d+1} \quad \text{in} \quad \Rdd;$ \\

\item[$b)$] $\displaystyle \frac{\partial u_j}{\partial x_k} = \frac{\partial u_k}{\partial x_j}, \quad j=1, \dots, d, \quad \text{in} \quad \Rdd;$\\

\item[$c)$]  $\displaystyle \frac{\partial u_{d+1}}{\partial x_j} = -i \, \partial_t^{1/2} u_{j}, \quad j=1, \dots, d, \quad \text{in} \quad \Rdd.$
\end{itemize}
Here, $\partial_t^{1/2}$ is the Weyl's fractional derivative operator of order $1/2$ with respect to $t$. If $g$ is a smooth enough function on $(0,\infty)$, $\partial_t^{1/2}g$ is defined by
\begin{equation*}\label{eq:Weyl}
(\partial_t^{1/2}g)(t)
:= \frac{e^{i\pi/2}}{\sqrt{\pi}} \int_t^\infty \frac{g'(s)}{\sqrt{s-t}} \, ds, \quad t>0.
\end{equation*}

In this paper we study solutions of the generalized temperature Cauchy-Riemann equations in amalgam spaces.\\

If $0<p,q<\infty$ we say that the vector $F:=(u_1, \dots, u_{d+1})$ of functions in $\Rdd$ is in $\HHpq$ when $F \in TCR(\Rdd)$ and
$$\|F\|_{\HHpq}
:=\sup_{t>0} \Big\| |F(\cdot,t)| \Big\|_{p,q}<\infty.$$

By using Riesz transforms we prove that the spaces $\Hpq$ and $\HHpq$ are topologically isomorphic. \\

We define the heat kernel
$$W_t(x)
:=\frac{e^{-|x|^2/4t}}{(4 \pi t)^{d/2}}, \quad x \in \Rd, \, t>0.$$

\begin{Th}\label{Th:1.3}
Let $(d-1)/d < min\{p,q\}$. We define, for every
$f \in \Hpq$,
$$\LL(f)(x,t)
:= \Big(
(R_1f*W_t)(x)
, \dots,
(R_nf*W_t)(x) ,
(f*W_t)(x)
\Big),
\ x \in \Rd, \ t>0.$$
Then, $\LL$ defines a topological isomorphism from $\Hpq$ onto $\HHpq$.
\end{Th}

This result is an extension of \cite[Theorem 2.10]{GP} to amalgam settings.\\

We also obtain characterizations of our Hardy spaces $\Hpq$ by using Fourier multipliers. We are motivated by the results in \cite[\S 16]{Uchib} (see also \cite{CCYY} and \cite{Uchi}).\\

By $\mathbb{S}^{d-1}$ we denote the unit sphere in $\Rd$. If $\theta \in L^\infty(\mathbb{S}^{d-1})$ we define the Fourier multiplier $m_\theta$ by
$$m_\theta(f)
:= \Big( \theta\Big( \frac{y}{|y|} \Big) \widehat{f} \Big)^{\vee},
\quad f \in L^2(\Rd).$$
Thus, $m_\theta$ is a bounded operator from $L^2(\Rd)$ into itself. Here, as usual, $\hat{g}$ and $\check{g}$ denotes the Fourier and inverse Fourier transform of $g$, respectively. It is clear that, for every $j=1, \dots, d$, the $j$-th Riesz transform $R_j$ coincides with $m_{\alpha_j}$ when $\alpha_j(z):=-i z_j$, $z \in \mathbb{S}^{d-1}$.\\

If $\theta \in C^\infty(\mathbb{S}^{d-1})$, then $m_\theta$ can be extended from
$L^2(\Rd) \cap \Hpq$ to $\Hpq$ as a bounded operator from $\Hpq$ into itself (see Section \ref{Sect:4}).\\

Suppose that
$\theta_j \in C^\infty(\mathbb{S}^{d-1})$, $j=1, \dots, m$, for certain $m \in \N$. We write $\Theta :=\{\theta_1, \dots, \theta_m\}$. We say that a function
$f \in L^2(\Rd)$ is in $\HHpqt$, with $0<p,q<\infty$, when
$m_{\theta_j}(f) \in (L^p,\ell^q)(\Rd)$, $j=1, \dots, m$. We define $\|f\|_{\Hpqt}$
as follows
$$\|f\|_{\Hpqt}
:= \sum_{j=1}^m \| m_{\theta_j}(f)  \|_{p,q}, \quad f \in \Hpqt.$$
The Hardy space $\Hpqt$ is defined as the completion of $\HHpqt$ with respect to the norm
$\| \cdot \|_{\Hpqt}$.\\

In the next result we establish a characterization of $\Hpq$ in terms of families of Fourier multipliers.

\begin{Th}\label{Th:1.4}
Assume that $m\in \mathbb{N}$ and $\Theta :=\{\theta_1, \dots, \theta_m\} \subset C^\infty(\mathbb{S}^{d-1})$
is such that
$$\rank \left(
\begin{array}{ccc}
\theta_1(y) & \dots & \theta_m(y) \\
\theta_1(-y) & \dots & \theta_m(-y) \\
\end{array}
\right) =2, \quad y \in \mathbb{S}^{d-1}.$$
There exists $1/2 < p_0 < 1$ such that if $p_0<\min\{p,q\}$ and $f \in L^2(\Rd)$, then $f \in \Hpq$ if, and only if, $m_{\theta_j}(f) \in (L^p,\ell^q)(\Rd)$, $j=1, \dots, m$. Furthermore, for every $f \in \Hpq \cap L^2(\Rd)$, the quantities $\|f\|_{\Hpq}$
and $\|f\|_{\Hpqt}$ are equivalent. Hence, $\Hpq = \Hpqt$ algebraic and topologically.
\end{Th}

We use in the proof of Theorem \ref{Th:1.4} a method due to Uchiyama (\cite{Uchi} and \cite{Uchib}) which does not require subharmonic properties, in contrast with the arguments employed in the proof of Theorems \ref{Th:1.1} and \ref{Th:1.2}. A key result in the proof of Theorem \ref{Th:1.4} is \cite[Lemma 27.2]{Uchib}. Note that Theorems \ref{Th:1.1} and \ref{Th:1.2} are different from Theorem \ref{Th:1.4}. The characterization in Theorem \ref{Th:1.4} really holds only for functions in $L^2(\Rd) \cap \Hpq$. However, in Theorems \ref{Th:1.1} and \ref{Th:1.2}  we give characterizations for all distributions in $\Hpq$.\\

After this introduction in Section \ref{Sect:2} we prove Theorem \ref{Th:1.1} and \ref{Th:1.2}. Theorems \ref{Th:1.3} and \ref{Th:1.4} are proved in Section \ref{Sect:3}
 and \ref{Sect:4}, respectively.\\

Throughout this paper by $C$ we always denote a positive constant that might change from one line to another.

\section{Proof of Theorems \ref{Th:1.1} and \ref{Th:1.2}.}
\label{Sect:2}

In this section we present a proof of Theorems \ref{Th:1.1} and \ref{Th:1.2}. We adapt the classical proofs (\cite[III, \S 4]{St1}) to amalgam spaces $(L^p,\ell^q)(\Rd)$. We also use some of the ideas introduced in \cite{YZN} where  variable exponent Lebesgue spaces are considered.

\begin{Prop}\label{Prop:2.1}
Let $0<p,q<\infty$. Suppose that $u$ is a harmonic function on $\Rdd$. Then, $u^* \in (L^p,\ell^q)(\Rd)$ if, and only if, there exists $f \in \Hpq$ such that
$$u(x,t)
= (f * P_t)(x), \quad (x,t) \in \Rdd.$$
Furthermore, there exists $C>0$ for which
$$\frac{1}{C} \|f\|_{\Hpq}
\leq \|u^*\|_{p,q}
\leq C \|f\|_{\Hpq}.$$
\end{Prop}

\begin{proof}
Assume that
$u(x,t)
= (f * P_t)(x)$, $(x,t) \in \Rdd$, for some $f \in \Hpq$.
Then, according to \cite[pp. 14 and 16]{AF1} we have that $f$ is a bounded distribution and
$u^* \in (L^p,\ell^q)(\Rd)$. Also,
$$\|u^*\|_{p,q}
\leq C \|f\|_{\Hpq}.$$

Suppose now that $u^* \in (L^p,\ell^q)(\Rd)$ and take $\eps>0$. The function
$$v_\eps(x,t):=u(x, t +\eps), \quad (x,t) \in \overline{\Rdd},$$
is bounded in $\overline{\Rdd}$. Indeed, by using H\"older inequality
(\cite[Proposition 11.5.4]{Heil}) we get, if $0<q \leq p$,
\begin{align*}
|u(x,t)|^q
& = \frac{1}{|B(x,t)|} \int_{B(x,t)} |u(x,t)|^q \, dy
\leq \frac{1}{|B(x,t)|} \int_{B(x,t)} |u^*(y)|^q \, dy \\
& \leq \frac{1}{|B(x,t)|} \||u^*|^q\|_{p/q,1} \, \|\1_{B(x,t)}\|_{(p/q)',\infty}
\leq \frac{C}{t^{nq/p}} \|u^*\|^q_{p,q},
\quad x \in \Rd, \, t>0,
\end{align*}
and, if $0<p \leq q$,
\begin{align*}
|u(x,t)|^p
& \leq \frac{1}{|B(x,t)|} \int_{B(x,t)} |u^*(y)|^p \, dy
 \leq \frac{1}{|B(x,t)|} \||u^*|^p\|_{1,q/p} \, \|\1_{B(x,t)}\|_{\infty,(q/p)'} \\
&\leq \frac{C}{t^{np/q}} \|u^*\|^p_{p,q},
\quad x \in \Rd, \, t>0.
\end{align*}
Hence,
$$|v_\eps(x,t)|
\leq \frac{C}{\eps^{n/\max\{p,q\}}} \|u^*\|_{p,q},
\quad x \in \Rd, \, t>0.$$

Furthermore, $v_\eps$ is continuous in $\overline{\Rdd}$ and harmonic in $\Rdd$. In particular, the function
$$f_\eps(x)
:=u(x,\eps), \quad x \in \Rd,$$
is continuous and bounded in $\Rd$. We consider the function
$$u_\eps(x,t)
:= P_t * f_\eps(x), \quad x \in \Rd, \, t>0.$$
Then, $u_\eps$ is continuous and bounded in $\overline{\Rdd}$, harmonic in $\Rdd$, and $u_\eps(x,0)=f_\eps(x)$, $x \in \Rd$.
According to \cite[Theorem 1.5, p. 40]{StWe} we conclude that
$u_\eps(x,t)=v_\eps(x,t)$, $x \in \Rd$ and $t \geq 0$.\\

By \cite[p. 16]{AF1} it follows that $f_\eps \in \Hpq$ and
$$\|f_\eps\|_{\Hpq}
\leq C \|P_*(f_\eps)\|_{p,q}
\leq C \|u^*\|_{p,q},$$
where the constant $C$ does not depend on $\eps$.\\

From \cite[Proposition 3.8 (1)]{AF1} the set $\{f_\eps\}_{\eps>0}$ is bounded in $S(\Rd)'$ when $S(\Rd)'$ is endowed with the weak--$*$ topology or with the strong topology.
Then, Banach-Alaoglu's Theorem (see, for example, \cite[pp. 68-70]{Rud}) implies that there exists $f \in S(\Rd)'$ such that
$$f_{\eps_k} \longrightarrow f, \quad \text{as } k \to \infty,
\quad \text{in } S(\Rd)',$$
for a certain decreasing sequence of positive numbers $\{\eps_k\}_{k \in \N}$ such that $\eps_k \to 0$, as $k \to \infty$. We recall that if $\{g_k\}_{k \in \N}$ is a sequence in $S(\Rd)'$ and $g \in S(\Rd)'$, then $g_k \longrightarrow g$, as $k \to \infty$, in the weak--$*$ topology of $S(\Rd)'$ if, and only if, $g_k \longrightarrow g$, as $k \to \infty$, in the strong topology of $S(\Rd)'$.\\

The arguments in \cite[p. 99]{St1} prove that there exists $\Phi \in S(\Rd)$ such that
$$M_\Phi(f_\eps)
\leq C u_\eps^*
\leq C u^*, \qquad \eps>0.$$
Then,
$\displaystyle \sup_{\eps>0} M_\Phi(f_\eps) \in (L^p,\ell^q)(\Rd)$. We have that
$$\lim_{k \to \infty} (f_{\eps_k}*\Phi_t)(x)
=(f*\Phi_t)(x), \quad x \in \Rd, \, t>0.$$
We get,
$$M_\Phi(f)
\leq \sup_{k \in \N} M_\Phi(f_{\eps_k}) \in (L^p,\ell^q)(\Rd).$$
From \cite[Theorem 3.7]{AF1} we deduce that $f \in \Hpq$.
As in \cite[p. 120]{St1} we also obtain that
$$u(x,t)
= \lim_{k \to \infty} u(x,t+\eps_k)
=(f * P_t)(x), \qquad x \in \Rd, \, t>0.$$
\end{proof}

We now establish some auxiliary results.

\begin{Lem}\label{Lem:2.2}
Let $0<p,q<\infty$ such that $(d-1)/d < \min\{p,q\}$.
Suppose that $F \in CR(\Rdd)$ and
$$\sup_{t>0} \| |F(\cdot,t)| \|_{p,q} < \infty.$$
Then, for any $\eta \in [(d-1)/d,\min\{p,q\})$ and $a \in (0,\infty)$,
$$|F(x,t+a)|
\leq \Big( (P_t*|F(\cdot,a)|^\eta)(x)\Big)^{1/\eta}, \quad x \in \Rd, \ t>0.$$
Furthermore, there exists a measurable function $g \geq 0$ on $\Rd$ such that
$$g(x)
:= \lim_{a \to 0^+} |F(x,a)|, \quad \text{a.e. } x \in \Rd,$$
and
$$|F(x,t)|
\leq \Big( (g^\eta * P_t)(x)\Big)^{1/\eta}, \quad x \in \Rd, \ t>0.$$
\end{Lem}

\begin{proof}
In order to prove this lemma we can proceed as in \cite[p. 122]{St1}. For $a>0$ we define
$$F_a(x,t)
:= F(x,t+a), \quad x \in \Rd, \, t>0.$$
According to \cite[Theorem VI.4.14]{StWe2}, $|F_a|^\delta$ is subharmonic in $\Rdd$ provided that $\delta \geq (d-1)/d$.
Let $\eta \in [(d-1)/d , \min\{p,q\})$. By using
\cite[Theorem 8]{Nu}, our lemma will be established once we guarantee that for some $r \in (1,\min\{p/\eta,q/\eta\})$
\begin{equation}\label{eq:R1}
\sup_{t>0} \int_{\Rd}
\frac{|F_a(x,t)|^{\eta r}}{(1+t+|x|)^{d+1}}
\, dx < \infty.
\end{equation}
Let $r \in (1,\min\{p/\eta,q/\eta\})$. By using H\"older inequality \cite[Proposition 11.5.4]{Heil} we get
\begin{align*}
 \int_{\Rd}
\frac{|F_a(x,t)|^{\eta r}}{(1+t+|x|)^{d+1}}
\, dx
& \leq \Big\| |F_a(\cdot,t)|^{\eta r}\Big\|_{p/(\eta r), q/(\eta r)} \,
\Big\| (1+t+|\cdot|)^{-(d+1)}  \Big\|_{(p/(\eta r))', (q/(\eta r))'} \\
& \leq \sup_{t>0} \Big\| |F_a(\cdot,t)|\Big\|^{\eta r}_{p, q} \,
\Big\| (1+|\cdot|)^{-(d+1)}  \Big\|_{(p/(\eta r))', (q/(\eta r))'}, \quad t>0.
\end{align*}
We have that
\begin{align}\label{eq:p7}
& \Big\| (1+|\cdot|)^{-d-1}  \Big\|^{(q/(\eta r))'}_{(p/(\eta r))', (q/(\eta r))'} \nonumber \\
& \qquad \leq C \int_{\Rd} \Big(
\int_{B(0,1)} (1+|x-y|)^{-(d+1)p/(p-\eta r)} dx
\Big)^{\frac{p-\eta r}{p} \frac{q}{q-\eta r}} \, dy \nonumber \\
& \qquad = C \Big\{ \int_{B(0,1)} \Big(
\int_{B(0,1)} (1+|x-y|)^{-(d+1)p/(p-\eta r)} dx
\Big)^{\frac{p-\eta r}{p} \frac{q}{q-\eta r}} \, dy \nonumber \\
& \qquad \qquad  + \int_{\Rd \setminus B(0,1)} \Big(
\int_{B(0,1)} (1+|x-y|)^{-(d+1)p/(p-\eta r)} dx
\Big)^{\frac{p-\eta r}{p} \frac{q}{q-\eta r}} \, dy
\Big\} \nonumber \\
& \qquad = C \Big\{ 1 + \int_{\Rd \setminus B(0,1)} \Big(
\int_{B(0,1)} (1+|x-y|)^{-(d+1)p/(p-\eta r)} dx
\Big)^{\frac{p-\eta r}{p} \frac{q}{q-\eta r}} \, dy
\Big\}.
\end{align}
Since
$$|y|
\leq |y-x| + |x|
\leq 1 + |x-y|, \quad x \in B(0,1), \quad y \in \Rd,$$
we get
\begin{align*}
& \int_{\Rd \setminus B(0,1)} \Big(
\int_{B(0,1)} (1+|x-y|)^{-(d+1)p/(p-\eta r)} dx
\Big)^{\frac{p-\eta r}{p} \frac{q}{q-\eta r}} \, dy \\
& \qquad \qquad \leq C
\int_{\Rd \setminus B(0,1)} |y|^{-(d+1)q/(q-\eta r)}  \, dy
\leq C.
\end{align*}
It follows that
$$\Big\| (1+|\cdot|)^{-d-1}  \Big\|_{(p/(\eta r))', (q/(\eta r))'}
< \infty,$$
and \eqref{eq:R1} holds.
\end{proof}

\begin{Prop}\label{Prop:2.3}
Let $(d-1)/d < min\{p,q\}<\infty$.
Suppose that $u$ is a harmonic function in $\Rdd$. Then,
$u^* \in (L^p,\ell^q)(\Rd)$ if, and only if, there exists a harmonic vector
$F:=(u_1, \dots, u_{d+1}) \in CR(\Rdd)$ such that $u_{d+1}:=u$ and
$$\sup_{t>0} \| |F(\cdot,t)| \|_{p,q} < \infty.$$
Furthermore, the quantities
$\displaystyle \sup_{t>0} \| |F(\cdot,t)| \|_{p,q}$
and
$\|u^*\|_{p,q}$
are equivalent.
\end{Prop}

\begin{proof}
This proposition can be proved in a similar way to
\cite[Proposition 2.6]{YZN}. The properties that we need will be specified in the sequel.\\

Assume that $u$ is harmonic in $\Rdd$ such that
$u^* \in (L^p,\ell^q)(\Rd)$. According to Proposition \ref{Prop:2.1} there exits $f \in \Hpq$ such that
$$u(x,t)
= (f * P_t)(x), \quad (x,t) \in \Rdd,$$
and
$$\|f\|_{\Hpq}
\leq C \|u^*\|_{p,q}.$$
By \cite[Theorem 3.11]{ZYYW}, for every $r>1$,
$L^r(\Rd) \cap \Hpq$ is a dense subspace of $\Hpq$. For every $j=1,\ldots,d$, and $f\in L^r(\Rd)$, $1\le r<\infty$, the j-th Riesz transform $\mathcal{R}_j(f)$ of $f$ is defined by
$$
\mathcal{R}_j(f)(x):=\lim_{\eps\to 0^+}\int_{|x-y|>\eps}K_j(x-y)f(y)dy, \quad a.\,e.\quad x\in \Rd,
$$
where
$$K_j(x)
:= \frac{\G((d+1)/2)}{\pi^{(d+1)/2}} \frac{x_j}{|x|^{d+1}}, \quad x \in \Rd \setminus \{0\}.$$
If $\phi,\varphi\in S(\Rd)$, by using Fourier transformation it is not hard to see that $\phi_t*\varphi\to 0$, as $t\to\infty$, in $L^2(\Rd)$. Hence, for every $g\in L^2(\Rd)$ and $\phi\in S(\Rd)$, $g*\phi_t\to 0$, as $t\to\infty$, in $S(\Rd)'$. By using \cite[Theorem 3.17]{ZYYW} and proceeding as in the proof of \cite[Theorem 16.1]{Uchib} we can prove that, for every $j=1,\ldots,d$, the j-th Riesz transform $\mathcal{R}_j$ can be extended from $L^2(\Rd) \cap \Hpq$ to $\Hpq$ as a bounded operator from $\Hpq$ into itself. This property was also proved in \cite[Corollary 4.19]{AF2} when $p\le q$. By taking into account that if
$\{g_k\}_{k=1}^\infty \subset (L^p,\ell^q)(\Rd)$ and
$$g_k \longrightarrow g, \quad \text{as } k \to \infty, \quad \text{in } (L^p,\ell^q)(\Rd),$$ there exists an increasing sequence $\{n_k\}_{k=1}^\infty \subset \N$ such that
$$g_{n_k}(x) \longrightarrow g(x), \quad \text{as } k \to \infty, \quad \text{a.e. } x \in \Rd,$$
the arguments in \cite[pp. 261-262]{YZN} allow us to show that
$$\sup_{t>0} \| |F(\cdot,t)| \|_{p,q}
\leq C \|u^*\|_{p,q},$$
where
$F:=(R_1 f*P_t, \dots, R_n f *P_t, f*P_t ) \in CR(\Rdd)$ .\\

Suppose now that $F:=(u_1, \dots, u_{d+1}) \in CR(\Rdd)$, $u_{d+1}:=u$, and
$$\sup_{t>0} \| |F(\cdot,t)| \|_{p,q}<\infty.$$
Let $\eta \in ((d-1)/d , \min\{p,q\})$. We have that
$$\sup_{t>0} \| |F(\cdot,t)|^\eta \|^{1/\eta}_{p/\eta,q/\eta}
= \sup_{t>0} \| |F(\cdot,t)| \|_{p,q}.$$
Since $(L^{p/\eta},\ell^{q/\eta})(\Rd)$ is a reflexive Banach space, there exists a decreasing sequence $\{t_k\}_{k \in \N} \subset (0,\infty)$ such that $t_k \longrightarrow 0$, as $k \to \infty$, and that
$$|F(\cdot,t_k)|^\eta \longrightarrow g, \quad \text{as } k \to \infty,$$
for a certain $0 \leq g \in (L^{p/\eta},\ell^{q/\eta})(\Rd)$, in the weak topology of
$(L^{p/\eta},\ell^{q/\eta})(\Rd)$. For every $x \in \Rd$ and $t>0$ we define
$$h_{x,t}(y)
:= P_t(x-y), \quad y \in \Rd.$$
By proceeding as in \eqref{eq:p7} we can see that
$h_{x,t} \in (L^{(p/\eta)'},\ell^{(q/\eta})')(\Rd)$. Hence we get
$$\lim_{k \to \infty} |F(\cdot,t_k)|^\eta * P_t(x)
=(g*P_t)(x), \quad x \in \Rd, \, t>0.$$
By  Lemma \ref{Lem:2.2} it follows that
$$ |F(x,t)|^\eta
\leq (g*P_t)(x), \quad x \in \Rd, \, t>0.$$
According to \cite[Proposition 11.12]{L+} the centered Hardy-Littlewood maximal function $\mathcal{M}$ is bounded from $(L^{p/\eta},\ell^{q/\eta})(\Rd)$ into itself. Then,
by using duality as in \cite[p. 261]{YZN} we obtain
$$\|u^*\|_{p,q}
\leq C \overline{\lim_{k \to \infty}} \| |F(\cdot,t_k)|^\eta \|^{1/\eta}_{p/\eta,q/\eta}
\leq C \sup_{t>0} \| |F(\cdot,t)| \|_{p,q}.$$
\end{proof}

Theorem \ref{Th:1.2} follows after combining Propositions \ref{Prop:2.1} and \ref{Prop:2.3}. Next, we concentrate on Theorem \ref{Th:1.1}.

\begin{proof}[Proof of Theorem \ref{Th:1.1}, $i)$]
Suppose firstly that $f \in S(\Rd)'$ satisfies the following condition:
\begin{equation}\label{eq:p8}
\text{for every } \Psi \in S(\Rd), \,
f * \Psi \in (L^{\mu p},\ell^{\mu q})(\Rd),
\text{ when } \mu \geq \mu_0,
\text{ for some } \mu_0 > \max\{1,1/p,1/q\}.
\end{equation}

We now define the Riesz transforms $R_j(f)$ of $f$ for every $j=1, \dots, d$.
Let $j=1, \dots, d$. We consider as above the kernel
$$K_j(x)
:= \frac{\G((d+1)/2)}{\pi^{(d+1)/2}} \frac{x_j}{|x|^{d+1}}, \quad x \in \Rd \setminus \{0\}.$$
We decompose $K_j$ as follows $K_j=:K_{j,1} + K_{j,2}$, where $K_{j,1}:=K_j \1_{B(0,1)}$.
The function $K_{j,1}$ defines a tempered distribution $S_j$ by
\begin{align*}
\langle S_j, \Psi \rangle
& = \lim_{\eps \to 0^+} \int_{\eps <|y|<1} K_j(y) \Psi(y) \, dy
  = \lim_{\eps \to 0^+} \int_{\eps <|y|<1} K_j(y) (\Psi(y)-\Psi(0)) \, dy \\
& =  \int_{B(0,1)} K_j(y) (\Psi(y)-\Psi(0)) \, dy , \quad \Psi \in S(\Rd).
\end{align*}
Note that the last integral is absolutely convergent for every $\Psi \in S(\Rd)$. Since the tempered distribution $S_j$ has compact support, $f*S_j$ is also a tempered distribution. We write $R_{j,1}:=f*S_j$.\\

Let $m \in \N$. We denote by $\OO_{c,m}$ the function space consisting on all those $h \in C^\infty(\Rd)$ such that, for every $k=(k_1, \dots, k_d) \in \N^d$,
$$\gamma_{m,k}(h)
:= \sup_{x \in \Rd} (1+|x|^2)^{-m}
\Big| \frac{\partial^{k_1 + \dots + k_d}}{\partial x_1^{k_1} \dots \partial x_d^{k_d}} h(x)\Big|
< \infty.$$
$\OO_{c,m}$ is endowed with the topology generated by the system of seminorms $\{\gamma_{m,k}\}_{k \in \N^d}$. Thus, $\OO_{c,m}$ is a Fr\'echet space.
According to \cite[proof of Proposition 7, p. 420]{Hor}, there exists $m_0 \in \N$ such that $f*\Psi \in \OO_{c,m_0}$, for every $\Psi \in S(\Rd)$, and the operator $\Psi \longmapsto f*\Psi$ is bounded from $S(\Rd)$ into $\OO_{c,m_0}$. Since the convergence in $\OO_{c,m_0}$ implies the pointwise convergence in $\Rd$, by using the assumption \eqref{eq:p8}, we can see that the operator $\Psi \longmapsto f*\Psi$ is closed and hence bounded from $S(\Rd)$ into
$(L^{\mu p},\ell^{\mu q})(\Rd)$, for every $\mu \geq \mu_0$.\\

By proceeding as in \eqref{eq:p7} we can see that $K_{j,2} \in (L^{r},\ell^{s})(\Rd)$ provided that $s>1$ and $r>0$. Then,
$K_{j,2} \in [(L^{\mu_0 p},\ell^{\mu_0 q})(\Rd)]'$. We define
$$\langle R_{j,2}f, \Psi \rangle
:= \int_{\Rd} (f * \Psi)(x) K_{j,2}(x) \, dx, \quad \Psi \in S(\Rd). $$
The above comments prove that $R_{j,2}f$ defines a tempered distribution.\\

We now define the Riesz transform $R_j f$ as the following tempered distribution
$$R_j f
:= R_{j,1}f + R_{j,2}f.$$
Let $\Psi \in S(\Rd)$. By taking into account the properties of the convolution of tempered distributions we can write
\begin{align*}
R_{j,1}f * \Psi
& = (f * S_j) * \Psi
= f * (S_j * \Psi)
= f * (\Psi* S_j)
= (f * \Psi)* S_j
= R_{j,1}(f * \Psi).
\end{align*}
We also have that
\begin{align*}
\langle R_{j,2}f * \Psi, \phi \rangle
& = \langle R_{j,2}(f) ,  \Psi * \phi \rangle
= \int_{\Rd} ( f * (\Psi * \phi) )(x) K_{j,2}(x) \, dx \\
& = \int_{\Rd} (( f * \Psi) * \phi)(x) K_{j,2}(x) \, dx
= \langle R_{j,2}(f * \Psi) ,  \phi \rangle, \quad \phi \in S(\Rd).
\end{align*}
Thus, we obtain
\begin{equation}\label{eq:R2}
(R_{j}f) * \Psi
=R_{j}(f * \Psi),
\end{equation}
where the equality is understood in $S(\Rd)'$.\\

Assume now that $g \in (L^{r},\ell^{s})(\Rd)$, for some $1<r,s<\infty$.
According to Young inequality on amalgams (\cite[Theorem 1.8.3]{Heil})
and \cite[(2.4)]{AF1} we deduce that
$g * \phi \in (L^{\alpha},\ell^{\beta})(\Rd)$, for every $\phi \in S(\Rd)$, $\alpha \geq r$ and $\beta \geq s$. Also by \cite[(2.5)]{AF1}, $g \in S(\Rd)'$. As above, we define
$R_j(g):=R_{j,1}(g) + R_{j,2}(g).$\\

Let $\Psi \in S(\Rd)$. We have that
$$\langle R_j(g), \Psi \rangle
= \langle g, S_j * \Psi \rangle
+ \int_{\Rd} (g*\Psi)(x) K_{j,2}(x) \, dx. $$
Since $K_{j,2} \in (L^{\alpha},\ell^{\beta})(\Rd)$, for every $\beta > 1$ and $\alpha >0$, we can write
$$\int_{\Rd} (g*\Psi)(x) K_{j,2}(x) \, dx
= [(g*\Psi)*K_{j,2}](0)
= \int_{\Rd} g(y) (\Psi*K_{j,2})(y) \, dy. $$
Also, we get
$$(S_j*\Psi)(x) + (\Psi*K_{j,2})(x)
= \lim_{\eps \to 0} \int_{|x-y|>\eps} \Psi(y) K_j(x-y) \, dy
=: \RR_j(\Psi)(x), \quad x \in \Rd.$$
Here $\RR_j$ denotes the $j$-th Riesz transform. Then,
$$\langle R_j g, \Psi \rangle
= \int_{\Rd} g(y) \RR_j(\Psi) \, dy.$$
As it is well-known $\RR_j$ is a Calder\'on-Zygmund operator associated with the kernel $K_j$. According to \cite[Theorem 7]{KNTYY} $\RR_j$ can be extended from $L^2(\Rd) \cap (L^{r},\ell^{s})(\Rd)$ as a bounded operator from $(L^{r},\ell^{s})(\Rd)$ into itself.
By taking into account that $S(\Rd)$ is a dense subspace of $(L^{r},\ell^{s})(\Rd)$, we obtain
$$\langle R_j g, \Psi \rangle
= - \int_{\Rd} \RR_j(g)(x) \Psi(x) \, dx.$$
We conclude that $R_j g = -\RR_j(g)$, in the sense of equality in $S(\Rd)'$.\\

By using \eqref{eq:R2} we deduce that $(R_j f) * \Psi \in (L^{\mu p},\ell^{\mu q})(\Rd)$,
when $\mu > \mu_0$ and $\Psi \in S(\Rd)$.\\

Suppose now that $f$ also satisfies the following property:
\begin{equation*}\label{eq:p10}
A:=
\sup_{\eps>0} \Big( \| f*\phi_\eps\|_{p,q}  + \sum_{j=1}^d \| R_j(f) * \phi_\eps \|_{p,q} \Big) < \infty.
\end{equation*}
Note that the above argument shows that, for every $j=1, \dots, d,$ and $\eps>0$,
$(R_j f)*\phi_\eps \in (L^{\mu p},\ell^{\mu q})(\Rd)$ when $\mu$ is large enough.\\

Let $\eps>0$. We define $f_\eps:=f*\phi_\eps$. Then, $f_\eps \in C^\infty(\Rd)$ and also $f_\eps \in L^\infty(\Rd)$. Indeed, suppose that $\Phi \in C_c^\infty(\Rd)$ is such that $\Phi(x)=1$, $|x|\leq 1$, and $\Phi(x)=0$, $|x|\geq 2$. Let $z \in \Rd$. We consider
$\Phi_z(x):=\Phi(x-z)$, $x \in \Rd$. We can write, for every
$x=(x_1, \dots, x_d) \in B(z,1)$,
\begin{equation}\label{eq:R3}
f_\eps(x)
= f_\eps(x) \Phi_z(x)
= \int_{-\infty}^{x_1} \dots \int_{-\infty}^{x_d} \frac{\partial^d}{\partial u_1 \dots \partial u_d} (f_\eps(u) \Phi_z(u)) \, du_1 \dots du_d.
\end{equation}
Let $\alpha=(\alpha_1, \dots, \alpha_d) \in \N^d$ and $\beta=(\beta_1, \dots, \beta_d) \in \N^d$. We denote
$$D^\alpha := \frac{\partial^{\alpha_1 + \dots + \alpha_d}}{\partial x_1^{\alpha_1} \dots \partial x_d^{\alpha_d}} \quad \text{and} \quad
D^\beta := \frac{\partial^{\beta_1 + \dots + \beta_d}}{\partial x_1^{\beta_1} \dots \partial x_d^{\beta_d}}.$$
By using H\"older inequality and \cite[(2.2)]{AF1} we deduce that
\begin{align*}
& \int_{-\infty}^{x_1} \dots \int_{-\infty}^{x_d}
|D^\alpha (f_\eps(u))| \,  |D^\beta (\Phi_z(u))| \, du_1 \dots du_d
\leq \int_{B(z,2)} |(f*D^\alpha \phi_\eps)(u)| \, du \\
& \qquad \qquad \leq C \| f * D^\alpha \phi_\eps \|_{\mu_0 p, \mu_0 q}
\Big( \int_{\Rd} \Big\| \1_{B(z,2)} \1_{B(y,1)} \Big\|_{(\mu_0p)'}^{(\mu_0 q)'} \, dy \Big)^{1/(\mu_0 q)'} \\
& \qquad \qquad \leq C \| f * D^\alpha \phi_\eps \|_{\mu_0 p, \mu_0 q}
\Big( \int_{B(z,3)} \, dy \Big)^{1/(\mu_0 q)'} 
\leq C \| f * D^\alpha \phi_\eps \|_{\mu_0 p, \mu_0 q}.
\end{align*}
From \eqref{eq:p8} and \eqref{eq:R3} it follows that $f_\eps \in L^\infty(\Rd)$.\\

We define $F_\eps:=(u_1^\eps, \dots, u_{d+1}^\eps)$, where for every $x \in \Rd$ and $t>0$,
$$
u_j^\eps(x,t):=(f_\eps * Q_t^j)(x), \quad j=1, \dots, d,
\qquad \text{and} \qquad
u_{d+1}^\eps(x,t):=(f_\eps * P_t)(x).
$$
Here $Q_t^j$ denotes the $j$-th conjugate Poisson kernel defined by
$$Q_t^j(x)
:= \frac{\G((d+1)/2)}{\pi^{(d+1)/2}} \frac{x_j}{(t^2+|x|^2)^{(d+1)/2}}, \quad x=(x_1, \dots, x_d) \in \Rd, \, t>0.$$
By \cite[Proposition 11.12]{L+} the centered Hardy-Littlewood maximal operator is bounded from $(L^{r},\ell^{s})(\Rd)$ into itself, provided that $1<r,s<\infty$. Then, by proceeding as in \cite[pp. 263--264]{YZN} we obtain that
$$\sup_{t>0} \| |F_\eps(\cdot,t)| \|_{p,q}
\leq C A.$$
Here, $C$ does not depend on $\eps$.\\

According to \cite[p. 90--91]{St1}, since $f$ is a bounded distribution, we can write
$$f*P_t
= (f*\Psi_t)*h_t + f * \Phi_t, \quad t>0,$$
where $\Psi$, $\Phi \in S(\Rd)$ and $h \in L^1(\Rd)$. By using vectorial Minkowski inequality and \cite[Proposition 2.1, (1d)]{AF1} we get
\begin{align*}
\|f*P_t\|_{\mu p, \mu q}
& \leq \int_{\Rd}  \|(f*\Psi_t)(\cdot-y)\|_{\mu p, \mu q} \, |h_t(y)| \,  dy
+ \|f*\Phi_t\|_{\mu p, \mu q} \\
& \leq   C\| f*\Psi_t\|_{\mu p, \mu q} \, \|h_t\|_1 + \|f*\Phi_t\|_{\mu p, \mu q}
< \infty, \qquad t>0 \text{ and } \mu \geq \mu_0.
\end{align*}
Also, $(f*P_t)(x)$, $x \in \Rd$ and $t>0$, is a bounded and harmonic function on $\Rdd$
(\cite[p. 90]{St1}). We can write
\begin{align*}
f_\eps * P_t
& = (f_\eps * \Psi_t)*h_t + f_\eps * \Phi_t \\
& = (f * (\phi_\eps*\Psi_t))*h_t + (f*\phi_\eps) * \Phi_t \\
& = ((f*\Psi_t) *\phi_\eps)*h_t + (f*\Phi_t ) * \phi_\eps \\
& = [((f*\Psi_t) *h_t) + f*\Phi_t]*\phi_\eps \\
& = (f * P_t)*\phi_\eps, \quad t>0.
\end{align*}
These equalities can be justified by using distributional Fourier transformation.
It follows that, for every $x \in \Rd$ and $t>0$,
$$(f_\eps * P_t)(x)
\longrightarrow (f*P_t)(x), \quad \text{as } \eps \to 0^+.$$

Let $1<r,s<\infty$. By proceeding as in the proof of \eqref{eq:R1} and using \cite[Proposition 2.1, (1d)]{AF1} we get, for every $j=1, \dots, d$, $x \in \Rd$ and $t>0$,
$$\int_{\Rd} |g(y)| \, |Q_t^j(x-y)| \, dy
\leq C \|g\|_{r,s} \|Q_{t}^j\|_{r',s'}
\leq C \|g\|_{r,s}, \quad g \in (L^{r},\ell^{s})(\Rd),$$
and by taking into account \cite[Theorem 7]{KNTYY},
$$\int_{\Rd} |\RR_j g(y)| \, |P_t(x-y)| \, dy
\leq C \|g\|_{r,s} \|P_{t}^j\|_{r',s'}
\leq C \|g\|_{r,s}, \quad g \in (L^{r},\ell^{s})(\Rd).$$
Here $C$ might depend on $t$ but it is $x$-independent. It is well-know that
$$(\Psi*Q_t^j)(x)
= (\mathcal{R}_j(\Psi)*P_t)(x), \quad x \in \Rd, \, t>0,$$
for every $\Psi \in S(\Rd)$ and $j=1, \dots, d$. Since $S(\Rd)$ is dense in
$(L^{r},\ell^{s})(\Rd)$ (\cite[(2.4)]{AF1}) we deduce that
$$(g*Q_t)(x)
= (\mathcal{R}_j(g)*P_t)(x), \quad x \in \Rd, \, t>0,$$
for every $g \in (L^{r},\ell^{s})(\Rd)$ and $j=1, \dots, d$.\\

By \eqref{eq:R2} we can write, for every $j=1, \dots, d$,
$$u_j^\eps(x,t)
= (R_j(f_\eps)*P_t)(x)
= \Big((R_j(f)*\phi_\eps)*P_t\Big)(x), \quad x \in \Rd, \, t>0.$$
Then,
$$u_j^\eps(x,t) \longrightarrow (R_j(f)*P_t)(x), \quad \text{as } \eps \to 0^+,$$
for every $x \in \Rd$, $t>0$ and $j=1, \dots d$.\\

We define $F:=(u_1,  \dots, u_{d+1})$, where, for every $x\in \Rd$ and $t>0$,
$$
u_j(x,t):=(R_j(f)*P_t)(x), \quad j=1, \dots, d,
\qquad
u_{d+1}(x,t):=(f*P_t)(x).
$$
By using Fatou's lemma we get
$$\sup_{t>0} \| |F(\cdot,t)| \|_{p,q}
\leq C A.$$
Then, Proposition \ref{Prop:2.3} implies that $u_{d+1}^* \in (L^{p},\ell^{q})(\Rd)$, that is,
$f \in \Hpq$. Furthermore, we have that
$$\|f\|_{\Hpq}
\leq C A.$$

We now prove the converse result. Suppose that $f \in \Hpq$. Let $\Psi \in S(\Rd)$. We have that $f * \Psi \in L^\infty(\Rd)$ (\cite[p. 14]{AF1}). By \cite[Theorem 3.7]{AF1},
$f * \Psi \in (L^{p},\ell^{q})(\Rd)$. If $\alpha \geq 1$ and $\|f*\Psi\|_\infty\neq 0$, we can write
\begin{align}\label{eq:R10}
\| f * \Psi \|_{\alpha p, \alpha q}
& = \Big( \int_{\Rd} \Big( \int_{B(y,1)} | (f*\Psi)(x) |^{\alpha p} \, dx \Big)^{q/p}dy \Big)^{1/\alpha q} \nonumber \\
& \leq \| f * \Psi \|_{\infty} \Big( \int_{\Rd} \Big(
\int_{B(y,1)} \Big(\frac{| (f*\Psi)(x) |}{\| f * \Psi \|_{\infty}}\Big)^{\alpha p} \, dx \Big)^{q/p} \Big)^{1/\alpha q} \nonumber \\
& \leq  \| f * \Psi \|_{\infty}^{1-1/\alpha} \| f * \Psi \|_{p,q}^{1/\alpha}.
\end{align}
Hence, $f * \Psi \in (L^{\alpha p},\ell^{\alpha q})(\Rd)$, $\alpha \geq 1$.\\

By combining the above estimates and \cite[(3.13)]{AF1} we obtain, for every $\alpha \geq 1$,
\begin{equation}\label{eq:R5}
\| f * \Psi \|_{\alpha p, \alpha q}
\leq C \|f\|_{\Hpq},
\end{equation}
where $C=C(\Psi,d,p,q)$.\\

As in the first part of this proof, we define, for every $j=1,\dots,d$, the Riesz transform $R_jf$ of $f$ by
$$R_jf:=
R_{1,j}f + R_{2,j}f.$$ 

We say that a Lebesgue measurable function $a$ defined in $\Rd$ is a $((L^p,\ell ^q)(\Rd),\infty,m)$--atom, where $m\in \mathbb{N}$, when there exists a cube $Q \subset \Rd$ such that 
\begin{itemize}
\item $supp(a)\subset Q$, 
\item $\displaystyle \int_Q a(x) \, x^\alpha \, dx=0$, for every $\alpha=(\alpha_1,\ldots,\alpha_d)\in \mathbb{N}^d$ with $\alpha_j\le m$, $j=1,\ldots,d$; and
\item $ \displaystyle 
\|a\|_\infty\le \|\1_Q\|^{-1}_{p,q}.
$
\end{itemize}
Here, $x^\alpha=\prod_{j=1}^dx_j^{\alpha_j}$ when $x=(x_1,\ldots,x_d)\in \Rd$ and $\alpha=(\alpha_1,\ldots,\alpha_d)\in \mathbb{N}^d$.\\

According to \cite[Theorem 3.11 and Proposition 3.15]{ZYYW} for $m$ large enough there exists a sequence
$\{a_n\}_{n \in \N}$ of $((L^p,\ell ^q)(\Rd),\infty,m)$--atoms in  $\Aa(p,\infty,\delta)$ and a sequence $\{\lambda_n\}_{n \in \N}\subset (0,\infty)$ such that
$$f
= \sum_{n=1}^\infty \lambda_n a_n,$$
where the convergence is understood in $S(\Rd)'$ or in $\Hpq$. We define
$$f_k
:=\sum_{n=1}^k \lambda_k a_k, \quad k \in \N.$$

Let $j=1,\dots,d$. Since $S_j$ is a distribution with compact support, we have that
$$\langle R_{j,1} f_k, \Psi \rangle
= \langle f_k, S_j * \Psi \rangle \longrightarrow \langle f,S_j*\Psi \rangle, \quad \text{as } k \to \infty,$$
for every $\Psi \in S(\Rd)$. Also by \eqref{eq:R5} we get
$$\langle R_{j,2} f_k, \Psi \rangle
= \int_{\Rd} (f_k*\Psi)(x) K_{j,2}(x)\, dx
\longrightarrow
\int_{\Rd} (f*\Psi)(x) K_{j,2}(x)\, dx
= \langle R_{j,2} f, \Psi \rangle, \quad \text{as } k \to \infty,$$
for every $\Psi \in S(\Rd)$. Hence,
$$ R_jf=\lim_{j\to\infty}R_{j} f_k
, \quad \text{in } S(\Rd)'.$$

On the other hand, according to \cite[Corollary 4.19]{AF2}, the classical Riesz transform $\RR_j$ defined by
$$\RR_jg(x)
:= \lim_{\eps \to 0^+} \int_{|x-y|>\eps} K_j(x-y) g(y) \, dy, \quad \text{a.e. } x \in \Rd,$$
for every $g \in L^2(\Rd)$, can be extended from
$L^2(\Rd) \cap \Hpq$ to $\Hpq$ as a bounded operator from $\Hpq$ into $\Hpq$. We have that
$$\RR_j f
= \lim_{k \to \infty} \RR_j f_k, \quad \text{in } \Hpq.$$
Then, by \cite[Proposition 3.8]{AF1}, we get
$$\RR_j f
= \lim_{k \to \infty} \RR_j f_k, \quad \text{in } S(\Rd)'.$$
Since, as it was proved above, $R_j f_k = \RR_j f_k$ because
$f_k \in L^2(\Rd) = (L^{2},\ell^{2})(\Rd)$, for every $k \in \N$, we conclude that
$R_j f = \RR_j f$.\\

By using \cite[Corollary 4.19]{AF2} it follows that
\begin{align*}
\|f * \phi_\eps\|_{p,q} + \sum_{j=1}^d \| R_j(f)*\phi_\eps \|_{p,q}
& \leq C \Big( \|f \|_{\Hpq} + \sum_{j=1}^d \| R_j(f) \|_{\Hpq}  \Big)
\leq C \|f \|_{\Hpq}.
\end{align*}
Thus, the proof is finished.
\end{proof}

\begin{proof}[Proof of Theorem \ref{Th:1.1}, $ii)$]
By taking into account the results established before the proof of Theorem \ref{Th:1.1}, $i)$, in order to prove Theorem \ref{Th:1.1}, $ii)$, we can proceed as in
\cite[pp. 123--124]{St1} (see also \cite[pp. 265--269]{YZN}). Note that, as it was proved before, if $f \in S(\Rd)'$ satisfies that $f * \Psi \in (L^{r},\ell^{s})(\Rd)$, for every
$\Psi \in S(\Rd)$, when $r$ and $s$ are large enough, then, for every $j=1, \dots, d$,
$R_j f \in S(\Rd)'$ and it has the same property. Then, an inductive argument shows that
$R_{j_1} \dots R_{j_m} (f) \in S(\Rd)'$ and it also satisfies the same property as $f$, for every $j_1, \dots, j_m \in \{1,\dots, d\}$ and $m \in \N$.
\end{proof}

\section{Proof of Theorem \ref{Th:1.3}}
\label{Sect:3}
Let $0<p,q<\infty$. We say that $u \in C^2(\mathbb R^{d+1}_+)$ is a temperature in $\mathbb R^{d+1}_+$, shortly $u \in T(\mathbb R^{d+1}_+)$, when
$$\frac{\partial u}{\partial t}
= \sum_{j=1}^d \frac{\partial^2 u}{\partial {x_i^2}}
\quad \text{in} \quad \mathbb R^{d+1}_+.$$
The space $T^{p,q}(\mathbb R^{d+1}_+)$ consists of all those $u \in T(\mathbb R^{d+1}_+)$ such that
$$
\|u\|_{T^{p,q}(\mathbb R^{d+1}_+)} := \sup_{t>0} \|u(\cdot,t)\|_{p,q} < \infty.
$$

\begin{Prop}\label{Prop:H1}
\quad
\begin{enumerate}
\item[$i)$] If $1 \leq p,q \leq \infty$ and $f \in (L^p,\ell^q)(\mathbb R^d)$, then
$$ u(x,t)
:= W_t(f)(x)
:= \int_{\Rd} W_t(x-y) f(y) \, dy,
\quad x \in \Rd, \, t>0,$$
is in $T^{p,q}(\mathbb R^{d+1}_+)$, and $$\|u\|_{T^{p,q}(\mathbb R^{d+1}_+)}\leq C\|f\|_{p,q}.$$
\item[$ii)$] If $1<p,q<\infty$ and $u \in T^{p,q}(\mathbb R^{d+1}_+)$, there exists $f \in (L^p,\ell^q)(\mathbb R^d)$ such that $u(x,t)=W_t(f)(x),\; x \in \mathbb R^d$ and $t>0$, and
$$\|f\|_{p,q} \leq C\|u\|_{T^{p,q}(\mathbb R^{d+1}_+)}.$$
\end{enumerate}
The constant $C>0$ in $i)$ (in $ii)$) does not depend on $f$ (on $u$).
\end{Prop}

\begin{proof}
Suppose that $1\leq p,q \leq \infty$. If $1\le p,q<\infty$, we have that
\begin{align*}
\|W_t\|^q_{p,q}
& \leq  C \int_{\mathbb R^n}
\Big(\int_{B(y,1)} \Big(
\frac{e^{-|z|^2/(4t)}}{t^{d/2}} \Big)^p \, dz \Big)^{q/p} \, dy \\
&\leq \frac{C}{t^{dq/2}} \Big\{ \int_{B(0,2)} \Big(\int_{B(y,1)}
e^{-p |z|^2/(4t)} \, dz \Big)^{q/p} \, dy \\
& \qquad \qquad + \int_{\Rd \setminus B(0,2)}\Big(\int_{B(y,1)} e^{-p |z|^2/(4t)} \, dz \Big)^{q/p} \, dy \Big\}\\
&\leq \frac{C}{t^{dq/2}}
\Big\{1+ \int_{\Rd \setminus B(0,2)}
e^{-c |y|^2/t}\, dy \Big\}
< \infty, \quad t>0.
\end{align*}
Also, we get $\|W_t\|_{\infty,\infty}< \infty$, $t>0$,
$$
\|W_t\|_{p,\infty}
\leq C \Big\|
\Big\| \1_{B(y,1)}  \frac{e^{-|z|^2/(4t)}}{t^{d/2}} \Big\|_p \Big\|_\infty
\leq C \Big\| \frac{e^{-|z|^2/(4t)}}{t^{d/2}} \Big\|_p < \infty, \qquad t>0, \quad 1\le p<\infty,
$$
and
\begin{align*}
\|W_t\|^q_{\infty,q}
& \leq C \int_{\mathbb R^n}
\Big\| \1_{B(y,1)} \frac{e^{-|z|^2/(4t)}}{t^{d/2}}\Big\|^q_{\infty} \, dy \\
& \leq  C \Big\{ \int_{B(0,2)} \Big\|
\frac{e^{-|z|^2/(4t)}}{t^{d/2}}\Big\|^q_{\infty} \, dy + \int_{\Rd \setminus B(0,2)}\Big\| \1_{B(y,1)}
\frac{e^{-|z|^2/(4t)}}{t^{d/2}}\Big\|^q_{\infty} \, dy \Big\} \\
& \leq  \frac{C}{t^{dq/2}}
\Big\{ 1+ \int_{\Rd \setminus B(0,2)} e^{-c|y|^2/t} \, dy\Big\} < \infty, \qquad t>0, \quad 1\le q<\infty.
\end{align*}
According to \cite[Proposition 2.1, $(1d)$]{AF1}, $W_t(x-\cdot) \in (L^p,\ell^q)(\mathbb R^d)$,
for every $x\in \mathbb R^n$ and $t>0$.\\

By using H\"older's inequality (\cite[Proposition 11.5.4]{Heil}) and \cite[Theorem 1, $(i)$]{Fl} we deduce that $W_t(f)(x) \in T(\mathbb R^{d+1}_+)$, provided that $f \in (L^p, \ell^q)(\mathbb R^d)$.
Then, vectorial Minkowski inequality implies that
$u(x,t)= W_t(f)(x) \in T^{p,q}(\mathbb R^{d+1}_+)$ and $$\|u\|_{T^{p,q}(\mathbb R^{d+1}_+)}
\le C\|f\|_{p,q},  \quad f\in (L^p, \ell^q)(\mathbb R^d).$$

Assume now that $u \in T^{p,q}(\mathbb R^{d+1}_+)$ where $1 < p,q < \infty$. Since $[(L^{p'}, \ell^{q'})(\mathbb R^d)]'= (L^p,\ell^q)(\mathbb R^d)$
(\cite[Theorem 11.7.1]{Heil}) by using Banach-Alaouglu Theorem and \cite[Theorem 1, (ii)]{Fl} we can prove that there exists $f \in (L^p,\ell^q)(\mathbb R^d)$ such that $u(x,t)=W_t(f)(x)$, $x\in \mathbb R^d$ and $t>0$, and $$\|f\|_{p,q}
\le C\|u\|_{T^{p,q}(\mathbb R^{d+1}_+)}.$$
\end{proof}




Notice that, if $F=(u_1,\ldots,u_{d+1}) \in TCR(\R^{d+1}_+)$ then $u_j \in T(\R^{d+1}_+)$, $j=1,\ldots,d+1$.\\

As in \cite[p. 611]{GP} we define, for every $j=1,\ldots,d$,
$$S_j(x,t)
:= \mathcal{R}_j(W_t)(x), \quad x\in \Rd, \, t > 0.$$
We can write
$$
W_t(x) = \int_{\R^d}e^{-4\pi^2t|y|^2}\cos(2\pi y\cdot x) \, dy,  \quad x\in \R^d, \, t>0,
$$
and for every $j=1,\ldots,d$,
$$
S_j(x,t) = \int_{\R^d}\frac{y_j}{|y|}e^{-4\pi^2t|y|^2}\sin(2\pi y\cdot x) \, dy, \quad x\in \R^d,  \, t>0.
$$
By \cite[Proposition 2.2]{GP}, $(S_1,\ldots,S_d,W) \in TCR(\R^{d+1}_+)$.
By using \cite[Lemma 2.3]{GP} and by proceeding as in the proof of \cite[Proposition 2.4]{GP} we can see that, for every $f\in (L^p,\ell^p)(\R^d)$ with $1 \leq p,q < \infty$,
$$
\Big( f \ast S_1(x,t), \ldots, f\ast S_d(x,t),f\ast W_t(x) \Big) \in TCR(\R^{d+1}_+).
$$


Riesz transforms connect the spaces $\HHpq$ and $T^{p,q}(\R^{d+1}_+)$.

\begin{Prop}\label{Prop:H2}
\quad
\begin{enumerate}
\item[$i)$] If $(d-1)/d < p,q < \infty$ and $F=(u_1,\ldots, u_{d+1}) \in \HHpq$, then $$u:=u_{d+1} \in T^{p,q}(\R^{d+1}_+),$$ $$u_j(\cdot,t)
=\mathcal{R}_j(u(\cdot,t)), \quad t>0, \,  j=1,\ldots,d, $$ and
$$
\|u\|_{T^{p,q}(\R^{d+1}_+)} \leq C\|F\|_{\HHpq}.
$$

\item[$ii)$] If $1<p,q<\infty$,
$u \in T^{p,q}(\R^{d+1}_+)$, and
$$F(x,t)
:=\Big(\mathcal{R}_1(u(x,t)), \ldots, \mathcal{R}_d(u(x,t)),u\Big),\quad x\in \Rd \quad and \quad t>0,$$
then $F \in \HHpq$
and
$$
\| F\|_{\HHpq}
\leq C \|u\|_{T^{p,q}(\R^{d+1}_+)}.
$$
\end{enumerate}
The constant $C$ in $i)$ (in $ii)$) does not depend on $F$ (on $u$).
\end{Prop}

\begin{proof}[Proof of Proposition \ref{Prop:H2}, $i)$]
Let $(d-1)/d< p,q< \infty$ and suppose that
$F:=(u_1,\ldots,u_{d+1}) \in \HHpq$. We write $u:=u_{d+1}$. We have that $u \in T^{p,q}(\R^{d+1}_+)$. For $x=(x_1,\ldots,x_d)\in \R^{d}$ and $t>0$, consider the rectangle
$$R
:=\prod^d_{j=1}\left(x_j-\frac{\sqrt{t}}{2\sqrt 2}, x_j+\frac{\sqrt t}{2\sqrt 2}\right).$$

Assume first that $q \leq p$. According to \cite[Lemma 1.1 and Remark, p. 610]{GP} we have that
$$
|u(x,t)|^q
\leq  \frac{C}{t^{1+d/2}}
\int_{t/2}^t \int_R |u(y,s)|^q \, dy \, ds.
$$
By using H\"older inequality (\cite[Proposition 11.5.4]{Heil}) we get
\begin{align*}
|u(x,t)|^q
&\leq  \frac{C}{t^{1+d/2}}
\int_{t/2}^t \|u(\cdot,s)\|^q_{p,q}
\|\1_R\|_{(p/q)', \infty} \, ds
 \leq  C t^{-1-d/2+d/(2(p/q)')+1} \|u\|^q_{T^{p,q}(\R^{d+1}_+)}  \\
& \leq  C t^{-qd/(2p)} \|u\|^q_{T^{p,q}(\R^{d+1}_+)}.
\end{align*}

Suppose now that $p \leq q$. By proceeding as before we get
\begin{align*}
|u(x,t)|^p
& \leq  \frac{C}{t^{1+d/2}}
\int_{t/2}^t \int_{R} |u(y,s)|^p \, dy \, ds
\leq  \frac{C}{t^{1+d/2}}
\int_{t/2}^t \|u(\cdot,s)\|^p_{p,q}\|\1_R\|_{\infty, (q/p)'} \, ds\\
&\leq  C t^{-pd/(2q)}\|u\|^p_{T^{p,q}(\R_+^{d+1})}.
\end{align*}
We obtain
\begin{equation}\label{H1}
|u(x,t)| \leq  C t^{-d/(2\max\{p,q\})}\|u\|_{T^{p,q}(\R_+^{d+1})}.
\end{equation}

Let $t_0 \in (0,\infty)$. We define, for every $x \in \R^{d}$ and $t\geq 0$,
$$u_0(x,t):=u(x,t+t_0) \qquad \text{and} \qquad
u_{j,0}(x,t):=u_j(x,t+t_0), \quad j=1, \dots, d.$$
According to (\ref{H1}),
$$u_0(\cdot,0) \in L^\infty(\R^d)
\qquad \text{and} \qquad
u_{j,0}(\cdot,0) \in L^\infty(\R^d), \quad j=1, \dots, d.$$
Also,
$$u_0(\cdot,0) \in (L^p,\ell^q)(\R^d)
\qquad \text{and} \qquad
u_{j,0}(\cdot,0) \in (L^p,\ell^q)(\R^d), \quad j=1, \dots, d.$$
By proceeding as in \eqref{eq:R10} we obtain that,
for every $\alpha \geq 1$,
$$u_0(\cdot,0) \in (L^{\alpha p}, \ell^{\alpha q})(\R^d)
\qquad \text{and} \qquad
u_{j,0}(\cdot,0) \in (L^{\alpha p}, \ell^{\alpha q})(\R^d), \quad j=1, \dots, d.$$

By \eqref{H1},
$$ \sup_{t>0}\|u_0(\cdot,t)\|_\infty < \infty
\qquad \text{and} \qquad
\sup_{t>0}\|u_{j,0}(\cdot,t)\|_\infty<\infty, \quad j=1, \dots, d.$$
Then, according to \cite[Theorem 4 (i)]{Fl}, for
$x \in \R^d$ and $t>0$,
$$u_0(x,t)=W_t(u_0(\cdot,0))(x)
\qquad \text{and} \qquad
u_{j,0}(x,t)=W_t(u_j(\cdot,0))(x), \quad j=1, \dots, d.$$
We take $\alpha_0> \max\{1,1/p, 1/q\}$. Since $\alpha_0p>1$ and $\alpha_0q>1$, by using Lemma \ref{Lem:3.3} below we get
$$\mathcal{R}_j(u_0(\cdot,0)=u_{j,0}(\cdot,0),
\quad j=1, \ldots,d.$$
Then, for every $t>0$, there exists $E \subset \R^d$ such that $|E|=0$ and
$$
u_{j,0}(x,t)
=u_j(x,t+t_0)
= \mathcal{R}_j(u_0(\cdot,t)(x)
= \mathcal{R}_j(u(\cdot,t+t_0))(x), \quad x\in \R^d\setminus E,
$$
for every $j=1\ldots,d$. We obtain that, for every $s>0$,
$$
u_j(x,s)= \mathcal{R}_j(u(\cdot,s))(x), \quad \text{a.e. } x \in \R^d,
$$
for each $j=1,\ldots,d$.
\end{proof}

\begin{proof}[Proof of Proposition \ref{Prop:H2}, $ii)$]
Suppose that $u \in T^{p,q}(\R_+^{d+1})$, where $1<p,q<\infty$. By \cite[Theorem 7]{KNTYY}, $R_j$ defines a bounded operator from $(L^p,\ell^p)(\R^d)$ into itself, for every $j=1,\ldots,d$. Hence, $$ \sup_{t>0}\|\mathcal{R}_j(u(\cdot,t))\|_{p,q}
\leq C \sup_{t>0}\|u(\cdot,t)\|_{p,q}.$$
According to Proposition \ref{Prop:H1}, $ii)$, there exists $f\in (L^p,\ell^q) (\R^{d+1}_+)$ such that $$u(x,t) = W_t(f)(x), \quad x \in \R^d, \, t>0.$$
We have that, for every $j=1,\ldots,d$, $$\mathcal{R}_j(u(\cdot,t))= W_t(\mathcal{R}_jf), \quad t>0, \, j=1, \dots,d.$$
It follows that $$\Big(\mathcal{R}_1(u(\cdot,t)),\ldots,\mathcal{R}_d(u(\cdot,t)), u(\cdot,t)\Big) \in \HHpq.$$
\end{proof}

We now establish the result that we have previously used in the proof of Proposition \ref{Prop:H2}, $i)$.

\begin{Lem}\label{Lem:3.3}
Let $1<p,q<\infty$. Suppose that $f,f_1,\ldots,f_d \in (L^p,\ell^q)(\R^d)$. We define, for $x \in \R^d$ and $t>0$,
$$u(x,t):=W_t(f)(x)
\qquad \text{and} \qquad
u_j(x,t):=W_t(f_j)(x), \quad j=1, \dots, d.$$
If $(u_1, \ldots,u_d,u)\in TCR(\R^{d+1}_+)$, then $f_j=\mathcal{R}_jf$, $j=1,\ldots,d$.
\end{Lem}
\begin{proof}
We have that
\begin{equation}\label{L1}
\Big|\frac{\partial W_t(x)}{\partial t}\Big|
\leq C\frac{e^{-c |x|^2/t}}{t^{1+d/2}} \leq C \min \left\{\frac{1}{t^{1+d/2}}, \frac{1}{|x|^{d+2}}\right\}, \quad x\in \R^d, \, t>0.
\end{equation}
Then, according to \cite[Lemma 2]{KS} we deduce that
$$
|\partial_t^{1/2} W_t(x)|
\leq C \min \Big\{\frac{1}{t^{(d+1)/2}}, \frac{1}{|x|^{d+1}}\Big\}, \quad x\in \R^d, \, t>0.
$$
Also, we get
\begin{align*}
\int_0^\infty s^{-1/2} \int_{\R^d}
\Big|\frac{\partial W_{t+s}(y)}{\partial s}\Big| \, dy \, ds
&\leq  C \int_0^\infty s^{-1/2}\Big(\int_{B(0,\sqrt{s+t})} \frac{dy}{(t+s)^{(d+2)/2}} \\
& \qquad \qquad \qquad + \int_{\Rd \setminus B(0,\sqrt{s+t})} \frac{dy}{|y|^{d+2}} \Big)ds\\
&\leq  C\int_0^\infty \frac{ds}{s^{1/2}(t+s)} < \infty, \quad t>0.
\end{align*}
We can write
$$
\widehat{\partial_t ^{1/2}W_t}(x)
=\partial_t^{1/2}\widehat{W_t}(x)
= i|x|e^{-t|x|^2}, \quad x\in \R^d, \, t>0.
$$

On the order hand, we have that
$$
\frac{ \partial (\partial_t^{1/2}W_t(x))}{\partial t}
=\partial_t^{1/2} \Big(\frac{\partial W_t(x)}{\partial t}\Big), \quad x\in \R^d, \, t>0,
$$
because, since
$$\Big| \frac{\partial^2 W_t(x)}{\partial t^2}\Big|
\leq C \min \Big\{\frac{1}{t^{(d+4)/2}}, \frac{1}{|x|^{d+4}}\Big\}, \quad x\in \R^d, \, t>0,$$
it follows that
$$
\int_0^\infty s^{-1/2}
\Big| \frac{\partial^2 W_{t+s}(x)}{\partial s^2}\Big| \, ds < \infty, \quad x\in \Rd, \quad t>0.
$$
By using again \cite[Lemma 2]{KS} we obtain
\begin{equation}\label{L2}
\Big|\partial^{1/2}_t
\Big(\frac{\partial W_t(x)}{\partial t}\Big)\Big|
\leq C \min \Big\{\frac{1}{t^{(d+3)/2}}, \frac{1}{|x|^{d+3}}\Big\}, \quad x\in \R^d, \, t>0.
\end{equation}
By proceeding as above we deduce that,
for any  $x\in \R^d$ and $t>0$,
$$
\widehat{(\,\,\,\partial_t^{1/2} \partial_t^{1/2}W_t\,)}(x)
= \partial_t^{1/2}\widehat{(\partial_t^{1/2}W_t)}(x)
= \partial_t^{1/2} \partial_t^{1/2}\widehat{W_t}(x)
= -|x|^2 e^{-t|x|^2}
= \widehat{\partial_t W_t}(x).
$$
Hence,
$$\partial_t^{1/2}\partial_t^{1/2}W_t(x)
= \frac{\partial W_t(x)}{\partial t}, \quad x\in \R^d, \, t>0.$$

Let $j=1,\ldots,d$. We define
$$v_j(x,t)
:= \Big(\mathcal{R}_j(f)\ast W_t \Big)(x), \quad x\in \R^d, \, t>0.$$

We have that
$$\partial_t^{1/2}\Big((\mathcal{R}_j(f)-f_j)\ast W_t\Big)(x)=0,
\quad x \in \R^d, \, t>0,$$
because $(v_1,\dots,v_d,u) \in TCR(\R^{d+1}_+)$ and $(u_1,\ldots, u_d,u) \in TCR(\R^{d+1}_+)$.\\

It remains to show that if $g \in (L^p,\ell^q)(\R^d)$ such that $\partial_t^{1/2}(g \ast W_t)(x)=0$, $x\in \Rd$ and $t>0$, then $g= 0$.
Indeed, take $g \in (L^p, \ell^q)(\R^d)$ verifying $\partial^{1/2}_t(g\ast W_t)(x)=0$. By using H\"older inequality, \cite[Proposition 2.1]{AF1} and (\ref{L1}) we get
\begin{align*}
&\int_0^\infty s^{-1/2} \int_{\R^d} |g(x-y)| \,
\Big| \frac{\partial W_{t+s}(y)}{\partial s}\Big| \, dy \,  ds
\leq C \int_0^\infty s^{-1/2}\|g\|_{p,q}
\Big\| \frac{\partial W_{t+s}}{\partial s} \Big\|_{p',q'} \, ds\\
& \qquad \leq C \int_0^\infty s^{-1/2}
\Big[\int_{\R^d}\Big(\int_{B(y,1)}\min\Big\{\frac{1}{(t+s)^{(d+2)/2}}, \frac{1}{|z|^{d+2}}\Big\}^{p'}dz
\Big)^{q'/p'}dy\Big]^{1/q'} \, ds\\
& \qquad \leq C \int_0^\infty s^{-1/2}
\Big\{
\int_{B(0,\sqrt{t+s}+2)}\Big(\int_{B(y,1)}
\frac{1}{(t+s)^{(d+2)p'/2}} \, dz\Big)^{q'/p'} \, dy \\
& \qquad \qquad + \int_{\Rd \setminus B(0,\sqrt{t+s}+2)}\Big(\int_{B(y,1)} \frac{1}{|z|^{(d+2)p'}}dz\Big)^{q'/p'} \, dy \Big\}^{1/q'} \, ds\\
& \qquad \leq C \int_0^\infty
s^{-1/2}\Big\{ \frac{(\sqrt{s+t}+2)^d}{(t+s)^{(d+2)q'/2}}
+ \int_{\Rd \setminus B(0,\sqrt{t+s}+2)} \frac{1}{|y|^{(d+2)q'}} \, dy
\Big\}^{1/q'} \, ds\\
& \qquad \leq C \Big(\int_0^\infty
\frac{s^{-1/2}}{(t+s)^{(d+2)/2-d/(2q')}} \, ds+\int_0^\infty
\frac{s^{-1/2}}{(t+s)^{(d+2)/2}}ds\Big) < \infty, \quad x \in \R^d, \, t>0.
\end{align*}
Hence,
$$\partial_t^{1/2}(g \ast W_t)
= g \ast \partial_t^{1/2}W_t, \quad t>0.$$
By using (\ref{L2}) and proceeding in a similar way we get
$$
\partial_t^{1/2}
\partial_t^{1/2}(g\ast W_t)
= g \ast
\partial_t^{1/2}
\partial_t^{1/2} W_t
= g \ast \frac{ \partial W_t}{\partial t} =0, \quad t>0.
$$
Since $(L^p,\ell^q)(\Rd) \subset S(\Rd)'$ (\cite[(2.5)]{AF1}) and $\partial W_t / \partial t \in S(\R^d)$, by applying the interchange formula we obtain
$$
\widehat{g \ast \frac{\partial W_t}{\partial t}}
= \widehat{g} \, \widehat{\frac{\partial W_t}{\partial t}}
=- \widehat{g} \, |x|^2e^{-t|x|^2}=0,\;\; t>0.
$$
Here the Fourier transformation and the equalities are understood in $S(\R^d)'$.\\

For every $\psi \in C_c^\infty(\R^d)$ such that $0 \notin \supp \psi$ we have that
$$
\langle\widehat{g},\psi \rangle = \langle\widehat{g},|x|^2e^{-|x|^2}e^{|x|^2}|x|^{-2}\psi \rangle=0,
$$
because $e^{|x|^2}|x|^{-2}\psi \in S(\R^d)$. Then, $\supp \widehat{g} = \{0\}$. According to
\cite[Theorem 6.25]{Rud}, there exists $n \in \N$ and $c_k \in \C$, $k=(k_1,\dots,k_d) \in \N^d$ such that $k_\ell \leq n$, $\ell=1,2,\ldots,d$, and
$$
\widehat{g} =
\sum_{\substack{k=(k_1,\ldots, k_d)\in \N^d \\k_\ell\leq n, \, \ell=1,\dots,d}} c_k \delta^{(k)}.
$$
Then, $g$ is a polynomial of degree at most $n$. Since $g \in (L^p,\ell^q)(\R^d)$ we deduce $g=0$.
\end{proof}

An immediate consequence of Propositions \ref{Prop:H1} and \ref{Prop:H2} is the following.
\begin{Cor} Let $1<p,q<\infty$. Then, the following assertions are equivalent.
\begin{enumerate}
\item[(i)] $F=(u_1,\dots,u_d,u) \in \HHpq$.
\item[(ii)] There exists $f\in (L^p,\ell^q)(\R^{d+1}_+)$ such that, for every
$x\in \R^d$ and $t>0$,
$$u(x,t)=W_t(f)(x)
\qquad \text{and} \qquad
u_j(x,t)=W_t(\mathcal{R}_j(f))(x), \quad j=1, \dots, d.$$
\end{enumerate}
Furthermore,
$$
\frac{1}{C}\|F\|_{\HHpq}
\leq \|f\|_{p,q}
\leq C\|F\|_{\HHpq},
$$
where $C>0$ does not depend on $F$.
\end{Cor}

We now establish the boundary behavior of the elements of $T^{p,q}(\R^{d+1}_+)$. Next result completes Proposition \ref{Prop:H1}.

\begin{Prop}\label{H3}
Let $0<p,q<\infty$. If $u \in T^{p,q}(\R^{d+1}_+)$, there exists $f \in \Ss(\R^d)'$ such that
$u(x,t)=(f \ast W_t)(x)$.
Furthermore,
$$ \lim_{t\rightarrow 0^+}u(\cdot,t)=f,
\quad \text{in } S(\R^d)'.$$
\end{Prop}

\begin{proof}
Assume that $u \in T^{p,q}(\R^{d+1}_+)$. According to (\ref{H1}), for every $t>0$, $$u(\cdot,t) \in L^\infty(\R^d)\subset \Ss(\R^d)'.$$
Also, by (\ref{H1}) and \cite[Theorem 17]{Fl}, there exists $f \in \Ss(\R^d)'$ such that $u(\cdot,t) = f\ast W_t$, $t>0$. It is well-known that
$$f\ast W_t \longrightarrow f
\quad \text{as } t \rightarrow 0^+,
\quad \text{in } S(R^d)'.$$
\end{proof}


\begin{proof}[Proof of Theorem \ref{Th:1.3}]
Let $f\in \Hpq$. By using Theorems \ref{Th:1.1} and \ref{Th:1.2} we get
$$
\sup_{t>0}\Big(\|f \ast W_t\|_{p,q} + \sum_{j=1}^d \|R_j(f) \ast W_t\|_{p,q}\Big) \leq C\|u^*\|_{p,q},
$$
where $u(x,t)=P_t(f)(x)$, $x\in \Rd$ and $t>0$.
Also, by \eqref{eq:R2},
$$R_j(f)\ast W_t
= R_j(f\ast W_t), \qquad t>0,
\quad j=1, \dots, d.$$
We define, for $x\in \R^d$ and $t>0$,
$$
v(x,t):=(f\ast W_t)(x)
\qquad \text{and} \qquad
v_j(x,t):=R_j(f\ast W_t)(x),
\quad j=1,\ldots,d.$$
There exists $\mu_0>1$ such that
$f\ast W_t \in (L^{\mu_0p},\ell^{\mu_0q})(\R^d)$, for every $t>0$. Then,
$G:=(v_1,\ldots,v_{d},v)\in TCR(\R^{d+1}_+)$. Hence, $G\in \HHpq$. Moreover, Proposition \ref{H3} implies that $$v(\cdot,t) \longrightarrow f
\quad \text{as} \quad t \rightarrow 0^+,
\quad \text{in } S(\R^d)'$$
and $$\|G\|_{\HHpq}\le C\|f\|_{\Hpq}.$$

Assume now $G:=(v_1,\ldots,v_d,v_{d+1}) \in \HHpq$. According to Proposition \ref{Prop:H2}, $ii)$,
$$v_{d+1}\in T^{p,q}(\R^{d+1}_+)
\qquad \text{and} \qquad
v_j(\cdot,t)=R_j(v_{d+1}(\cdot,t)), \quad t>0, \, j=1,\ldots,d.$$
By Proposition \ref{H3}, there exists $f \in S(\R^d)'$ such that
$$v_{d+1}(x,t) = (f \ast W_t)(x),
\quad x\in \R^d, \, t>0,$$
and
$$v_{d+1}(\cdot,t) \longrightarrow f,
\quad \text{as} \quad t \rightarrow 0^+,
\quad \text{in } S(\R^d)'.$$
As it was shown in the proof of Proposition \ref{Prop:H2}, $ii)$,
$$v_{d+1}(\cdot,t) \in (L^{\mu p},\ell^{\mu q})(\R^d), \quad \mu >1 \text{ and } t>0.$$

Let $t_0>0$. We define
$$u^0_{d+1}(x,t)
:= (P_t\ast v_{d+1}(\cdot,t_0))(x), \quad x\in \R^d, \, t>0.$$
By taking $\mu_0 >\max\{1,1/p,1/q\}$ we have that
$$u^0_{d+1}(\cdot,t) \in (L^{\mu p},\ell^{\mu q})(\R^d), \quad \mu > \mu_0
\text{ and } t>0,$$
and by \cite[Proposition 2,(1d)]{AF1} and the vectorial Minkowski inequality, for every $\mu \geq \mu_0$ and $t>0$,
$$
\left\|u^0_{d+1}(\cdot,t)\right\|_{\mu p,\mu q} \leq \left\| v_{d+1}(\cdot,t_0)\right\|_{\mu p,\mu q}.
$$
We define
$$u^0_j(\cdot,t)
:= \mathcal{R}_j(u^0_{d+1}(\cdot,t)), \quad t>0, \,
j=1,\ldots,d.$$
We have that
$$u^0_j(\cdot,t)
= P_t \ast \mathcal{R}_j(v_{d+1}(\cdot,t_0)),
\quad t>0, \, j=1,\ldots,d.$$
Hence, for every $j=1,\ldots,d$,
$$u_j^0(x,t) \longrightarrow \mathcal{R}_j(v_{d+1}(\cdot,t_0))(x)
\quad \text{as } t \rightarrow 0^+,
\text{ for almost all } x \in \R^d.$$
We can write, for $x\in \R^d$ and $t>0$,
$$
v_{d+1}(x,t+t_0)
=(f \ast W_{t+t_0})(x)
= \left[(f\ast W_{t_0})\ast W_t\right](x) = (v_{d+1}(\cdot,t_0)\ast W_t)(x),
$$
and, for every $j=1,\ldots,d$,
$$
v_j(x,t+t_0)
= \mathcal{R}_j(v_{d+1}(\cdot,t_0) \ast W_t)(x) = (\mathcal{R}_j(v(\cdot,t_0))\ast W_t)(x), \quad x\in \Rd, \quad t>0.
$$
According to Theorems \ref{Th:1.1} and \ref{Th:1.2}, we get
\begin{align*}
\|(u_{d+1}^0)^*\|_{p,q}
& \leq  C \sup_{t>0}\Big(\|v_{d+1}(\cdot,t_0) \ast W_t\|_{p,q} + \sum_{j=1}^d\|R_j(v_{d+1}(\cdot,t_0))\ast W_t\|_{p,q}\Big)\\
&= C\sup_{t>0}\Big(\|v_{d+1}(\cdot,t+t_0) \|_{p,q} + \sum_{j=1}^d\|v_{j}(\cdot,t+t_0)\|_{p,q}\Big)\\
&\leq C \sup_{t>0}\||G(\cdot,t)|\|_{p,q}.
\end{align*}
Note that $C$ does not depend on $t_0$.\\

Also we have that
$$
\Big\|\sup_{t>0}|v_{d+1}(\cdot,t_0) \ast W_t|\Big\|_{p,q}
\simeq C \|(u^0_{d+1})^*\|_{p,q}
\leq C \sup_{t>0}\||G(\cdot,t)|\|_{p,q}.
$$
Then,
$$
\Big\|\sup_{t>t_0} |f\ast W_t|\Big\|_{p,q} \leq C \sup_{t>0}\||G(\cdot,t)|\|_{p,q}.
$$
It is clear that
$$
\sup_{t>a}|(f\ast W_t)(x)|
\le \sup_{t>b}|f \ast W_t)(x)|,\;\; x \in \R^d, \;\;  0<b \leq a.
$$
Then,
$$ \lim_{t_0 \rightarrow 0^+}\sup_{t>t_0}|(f \ast K_t)(x)| = \sup_{t>0}|(f\ast K_t)(x)|, \quad x\in \R^d.$$
By using monotone convergence theorem we conclude that
$$
\Big\|\sup_{t>0}|f\ast W_t|\Big\|_{p,q}
\leq C \sup_{t>0}\| G(\cdot,t)|\|_{p,q}.
$$
Therefore, $f \in \HH^{p,q}(\R^d)$.\\

The vector
$$F(x,t)= \Big((R_1(f)\ast P_t)(x), \ldots (R_d(f)\ast P_t)(x),(f\ast P_t)(x)\Big), \quad x\in \R^d, \, t>0,$$
is harmonic, satisfies the generalized Cauchy-Riemann equation, and
$$
\sup_{t>0}\||F(\cdot,t)|\|_{p,q} \leq C \sup_{t>0}\||G(\cdot,t)|\|_{p,q},
$$
(see Proposition \ref{Prop:2.3}).

\end{proof}

\section{Proof of Theorem \ref{Th:1.4}}
\label{Sect:4}


Let $j=1, \dots, m$. We define
$$K_j(g)
:= \Big( \theta_j\Big( \frac{y}{|y|} \Big) \widehat{g} \Big)^{\vee}, \quad g \in L^2(\Rd).$$
It is clear that $K_j$ is bounded from $L^2(\Rd)$ into itself. According to \cite[Theorem 3.11]{ZYYW}, $L^2(\Rd) \cap \Hpq$ is dense in $\Hpq$. By proceeding as in the proof of \cite[Theorem 16.1]{Uchib} we are going to see that $K_j$ can be extended to $\Hpq$ as a bounded operator from $\Hpq$ into itself. \\

Indeed, for every $\phi\in S(\Rd)$ and $f\in S(\Rd)'$ we define the area integral $S_\phi(f)$ by
$$
S_\phi(f)(x):=\Big(\int_{\Gamma(x)}|(\phi(\cdot \, t)\widehat{f})^{\vee}(y)|^2\frac{dydt}{t^{n+1}}\Big)^{1/2},\quad x\in \Rd,
$$
where $\Gamma(x):=\{(y,t)\in \mathbb{R}^{d+1}_+:\,|y-x|<t\}$, for every $x\in \Rd$.\\

Let $\varphi \in S(\Rd)$ such that
$\1_{B(0,4) \setminus B(0,2)} \leq \varphi \leq \1_{B(0,8) \setminus B(0,1)}$
and
$g \in L^2(\Rd) \cap \Hpq$. We can write
$$S_\varphi (K_j g)(x)
= S_{\phi_j}(g)(x), \quad x \in \Rd,$$
where $\phi_j(y):=\varphi(y) \theta_j(y/|y|)$, $y \in \Rd$. Note that $\phi_j \in S(\Rd)$ and $\supp \phi_j \subseteq \{x \in \Rd : 1 \leq |x| \leq 8\}$.\\

Since $g \in L^2(\Rd)$, $K_j(g) \in L^2(\Rd)$ and, for every $\Phi \in S(\Rd)$,
$$\Phi_t * K_j(g) \longrightarrow 0, \quad t \to \infty, \quad \text{in } S(\Rd)'.$$
Also, we have that
$$\| S_\varphi (K_j g) \|_{p,q}
= \| S_{\phi_j}(g) \|_{p,q}
\leq C \|g\|_{\Hpq}.$$
Then, according to \cite[Theorem 3.17]{ZYYW}, we conclude that $K_j(g) \in \Hpq$ and
$$\| K_j g \|_{\Hpq}
\leq C \|g\|_{\Hpq}.$$

Now as in \cite[Theorem 16.2]{Uchib} we deduce that, if $g \in L^2(\Rd)$ and
$K_j(g) \in \Hpq$, $j=1, \dots, m$, then $g \in \Hpq$ and
$$ \|g\|_{\Hpq}
\leq C \sum_{j=1}^m \| K_j g \|_{\Hpq}.$$

We define, for every $r>0$, the maximal function $\MM_r$ as follows
$$\MM_r(g)(x)
:= \sup_{t>0} \Big( \int_{B(x,r)} |g(y)|^r \, dy \Big)^{1/r}, \quad x \in \Rd,$$
for every $g \in L^1_{loc}(\Rd)$. By \cite[Proposition 11.12]{L+}, $\MM_r$ is bounded from $(L^\alpha,\ell^\beta)(\Rd)$ into itself, provided that $r<\alpha,\beta<\infty$.
Then, according to \cite[Lemma 27.2]{Uchib} and \cite[Theorem 3.7]{AF1}, there exists $p_0\in (1/2,1)$ such that
$$ \| K_j g \|_{\Hpq}\le C\Big\|\mathcal{M}_{p_0}\Big(\mathcal{M}_{1/2}\Big(\sum_{\ell=1}^m|K_\ell(g)|\Big)\Big)\Big\|_{p,q}
\leq C \sum_{\ell=1}^m \| K_\ell g \|_{p,q}, \quad g \in L^2(\Rd),$$
provide that $p,q>r_0$.\\

By combining the above estimates we conclude that
$$
\|f\|_{\Hpq}
\leq C \sum_{\ell=1}^m \| K_\ell f \|_{p,q},$$
for every $f \in L^2(\Rd) \cap \Hpq$.\\

On the other hand, if $f\in L^2(\Rd)\cap \Hpq$, then 
$$\lim_{t\to 0^+}P_t(K_\ell(f))(x)=K_\ell(f)(x), \quad \text{a.e. } x \in \Rd,$$
for every $\ell=1,\ldots,m$, and by using Fatou Lemma it follows that
$$
\sum_{\ell=1}^m \| K_\ell f \|_{p,q}\le C\|f\|_{\Hpq}.
$$
Hence, $\Hpq = \Hpqt$ algebraic and topologically, where $\Theta :=\{\theta_1, \dots, \theta_m\}$.


\end{document}